\documentclass[final,3p,times]{elsarticle}

\usepackage{amssymb}
\usepackage{amsmath}
\usepackage{amsthm}
\usepackage{subfigure}
\usepackage{graphicx}
\usepackage[colorlinks=true]{hyperref}

\newcommand{\bbZ}{\mathbb{Z}}
\newcommand{\bbR}{\mathbb{R}}

\newcommand{\rmp}{\mathrm{p}}
\newcommand{\rmE}{\mathrm{E}}
\newcommand{\rmP}{\mathrm{P}}
\newcommand{\rmT}{\mathrm{T}}

\newcommand{\calR}{\mathcal{R}}
\newcommand{\calW}{\mathcal{W}}
\newcommand{\calX}{\mathcal{X}}
\newcommand{\calY}{\mathcal{Y}}

\newcommand{\LH}{\mathrm{LH}}
\newcommand{\occ}[1]{\mathrm{occ}_{#1}}

\newcommand{\bigparen}[1]{\bigl({#1}\bigr)}
\newcommand{\Bigparen}[1]{\Bigl({#1}\Bigr)}
\newcommand{\biggparen}[1]{\biggl({#1}\biggr)}
\newcommand{\Biggparen}[1]{\Biggl({#1}\Biggr)}
\newcommand{\Biggbracket}[1]{\Biggl[{#1}\Biggr]}
\newcommand{\set}[1]{\{{#1}\}}

\newcommand{\abs}[1]{|{#1}|}
\newcommand{\bigabs}[1]{\bigl|{#1}\bigr|}
\newcommand{\biggabs}[1]{\biggl|{#1}\biggr|}
\newcommand{\Biggabs}[1]{\Biggl|{#1}\Biggr|}
\newcommand{\norm}[1]{\|{#1}\|}
\newcommand{\Bignorm}[1]{\Bigl\|{#1}\Bigr\|}
\newcommand{\ip}[2]{\langle{#1},{#2}\rangle}
\newcommand{\conv}[2]{{#1}*{#2}}
\newcommand{\sconv}[2]{{#1}\star{#2}}

\newcommand{\alphi}{\renewcommand{\labelenumi}{(\alph{enumi})}}

\newtheorem{thm}{Theorem}

\newtheorem{prop}[thm]{Proposition}

\theoremstyle{definition}

\begin{document}
\begin{frontmatter}
\title{Local histograms and image occlusion models}

\author[AFIT]{Melody L. Massar}
\author[CMUECE]{Ramamurthy Bhagavatula}
\author[AFIT]{Matthew Fickus}
\ead{Matthew.Fickus@afit.edu}
\author[CMUECE,CMUBME]{Jelena Kova\v{c}evi\'{c}}

\address[AFIT]{Department of Mathematics and Statistics, Air Force Institute of Technology, Wright-Patterson Air Force Base, OH 45433, USA}
\address[CMUECE]{Department of Electrical and Computer Engineering, Carnegie Mellon University, Pittsburgh, PA 15213, USA}
\address[CMUBME]{Department of Biomedical Engineering, Carnegie Mellon University, Pittsburgh, PA 15213, USA}

\begin{abstract}
The local histogram transform of an image is a data cube that consists of the histograms of the pixel values that lie within a fixed neighborhood of any given pixel location.  Such transforms are useful in image processing applications such as classification and segmentation, especially when dealing with textures that can be distinguished by the distributions of their pixel intensities and colors.  We, in particular, use them to identify and delineate biological tissues found in histology images obtained via digital microscopy.  In this paper, we introduce a mathematical formalism that rigorously justifies the use of local histograms for such purposes.  We begin by discussing how local histograms can be computed as systems of convolutions.  We then introduce probabilistic image models that can emulate textures one routinely encounters in histology images.  These models are rooted in the concept of image occlusion.  A simple model may, for example, generate textures by randomly speckling opaque blobs of one color on top of blobs of another.  Under certain conditions, we show that, on average, the local histograms of such model-generated-textures are convex combinations of more basic distributions.  We further provide several methods for creating models that meet these conditions; the textures generated by some of these models resemble those found in histology images.  Taken together, these results suggest that histology textures can be analyzed by decomposing their local histograms into more basic components.  We conclude with a proof-of-concept segmentation-and-classification algorithm based on these ideas, supported by numerical experimentation.

\end{abstract}

\end{frontmatter}

%%%%%%%%%%%%%%%%%%%%%%%%%%%%%%%%%%%%%%%%%%%%%%%%%%%%%%%%%%%%%%%%
%%%%%%%%%%%%%%%%%%%%%%%%%%%%%%%%%%%%%%%%%%%%%%%%%%%%%%%%%%%%%%%%
% Section 1: Introduction
%%%%%%%%%%%%%%%%%%%%%%%%%%%%%%%%%%%%%%%%%%%%%%%%%%%%%%%%%%%%%%%%
%%%%%%%%%%%%%%%%%%%%%%%%%%%%%%%%%%%%%%%%%%%%%%%%%%%%%%%%%%%%%%%%

\section{Introduction}

A \textit{local histogram} of an image is a histogram of the values of the pixels that lie in a neighborhood of a given pixel's location.  It indicates the particular combination of pixel intensities or colors that appear in that neighborhood.  When used as features in an image classification scheme, such histograms can help distinguish one texture from another.  We, in particular, use them in automated segmentation-and-classification algorithms for digital microscope images of biological tissues.

To be precise, the work presented here was motivated by the need to identify and delineate the various tissues exhibited in images of histological sections of teratoma tumors derived from embryonic stem cells, such as the one given in Figure~\ref{figure.Histology}(a).  This image was provided by Dr.~Carlos Castro of the University of Pittsburgh and Dr.~John A.~Ozolek of the Children's Hospital of Pittsburgh, who grow and image such teratomas to gain greater insight into tissue development.  In this image, which is purple-pink from hematoxylin and eosin (H\&E) staining, even a layman can discern several distinct textures, each corresponding to a distinct tissue type.  For each image under study, Drs. Castro and Ozolek make use of their years of medical training and experience to identify what tissues are present, and to what degree.  Moreover, when provided with a point-and-click interface, they can manually segment the image according to tissue type, resulting in per-pixel labels such as those given in Figure~\ref{figure.Histology}(b).  Though straightforward for medical experts, such tasks are nevertheless tedious and time-consuming, leading to inconsistencies when working with large data sets.  It is therefore our goal to automate as much of this process as is possible.  Our current algorithm is given in~\cite{Bhagavatula10Nov} and builds upon previous work given in~\cite{Bhagavatula10,Chebira08,Massar10}.
\begin{figure}
\begin{center}
\subfigure[]
{\includegraphics[width=.45\textwidth]{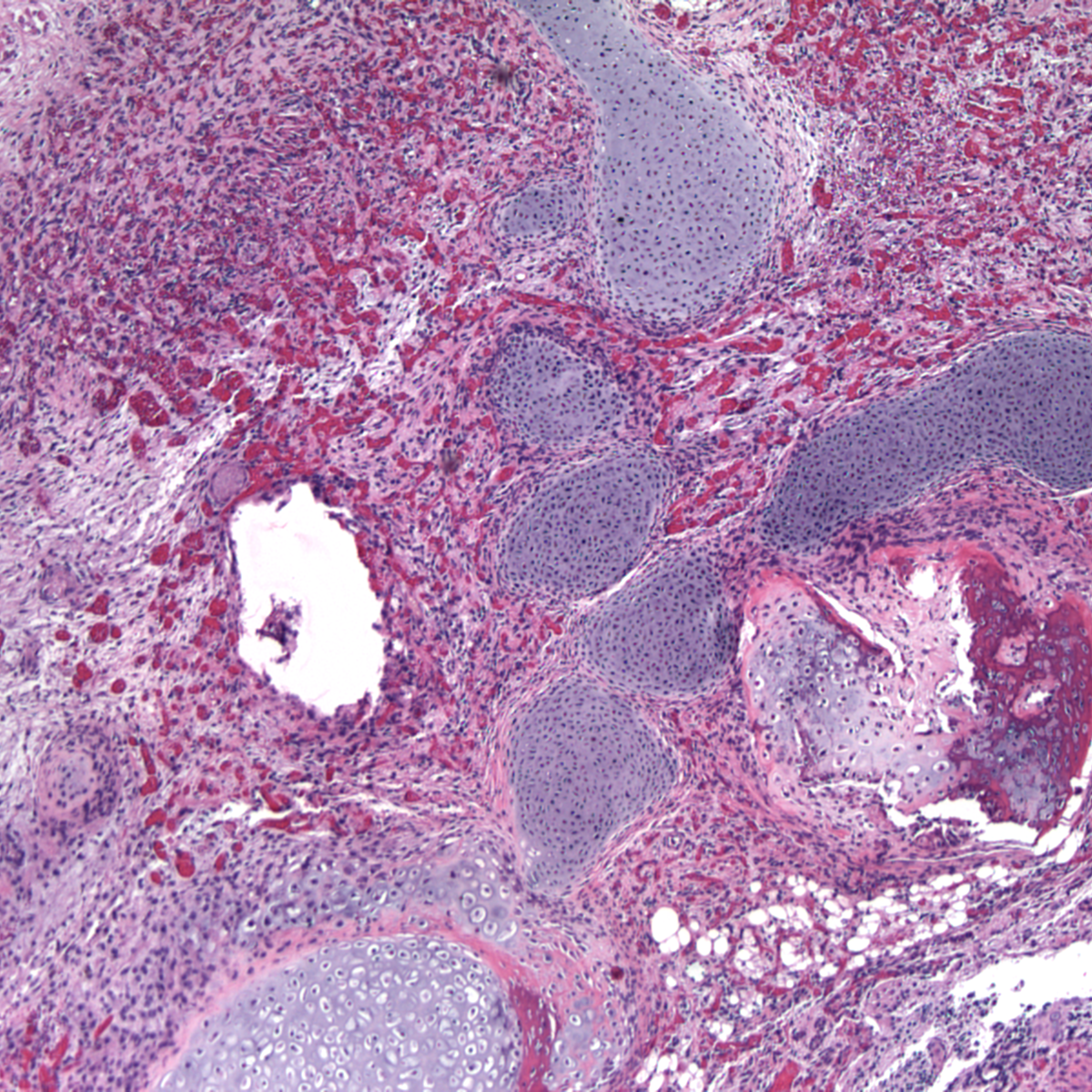}}
\hspace{.4in}
\subfigure[]
{\includegraphics[width=.45\textwidth]{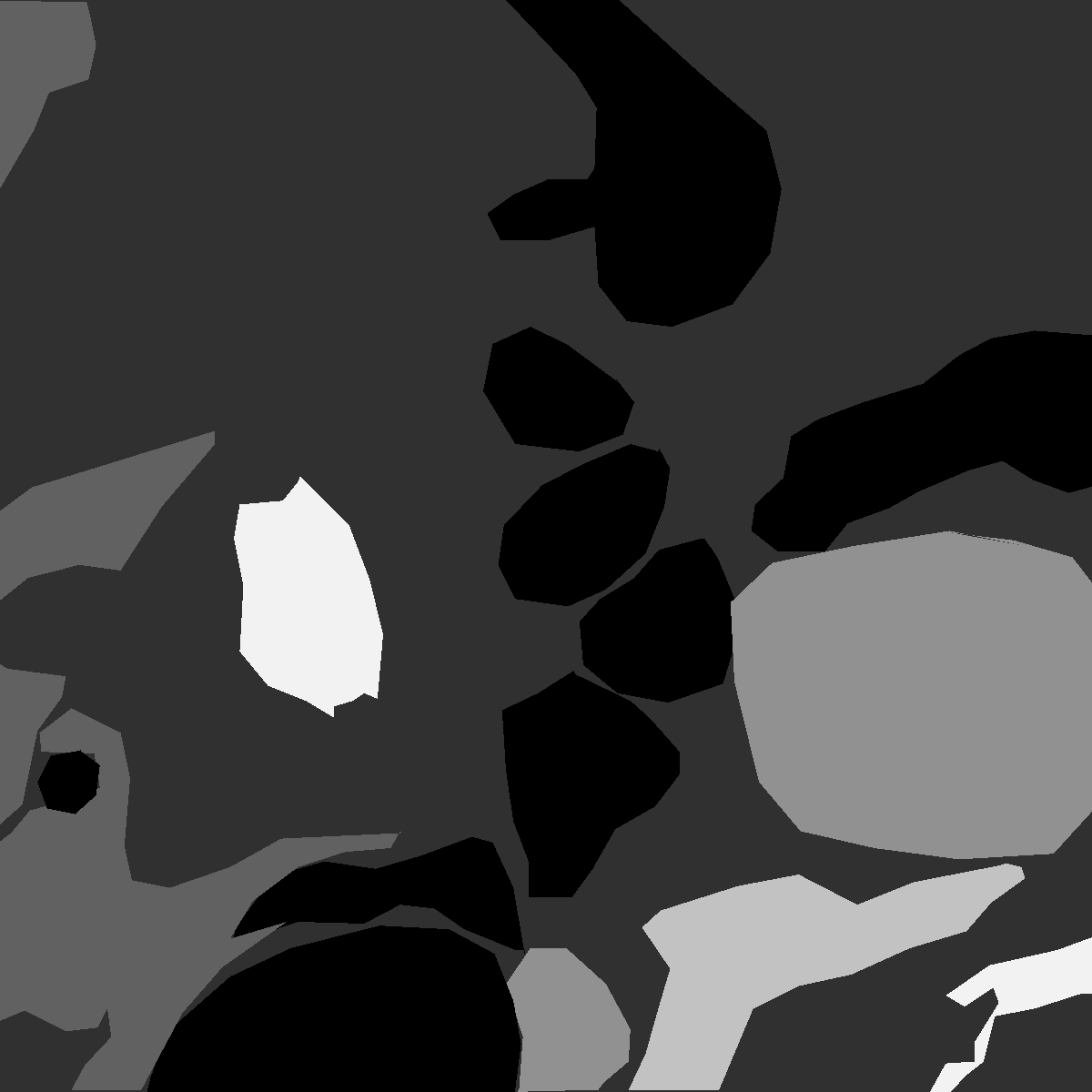}}
\caption{\label{figure.Histology} (a) A digital microscope image of a H\&E-stained tissue section.  (b) The histology image has been manually segmented and classified by a medical expert, resulting in the per-pixel labels.  From darkest to lightest, the labels indicate cartilage, pseudovascular tissue, connective tissue, bone, fatty tissue, and background pixels, respectively.  Our goal is to automate this segmentation-and-classification process.  The purpose of this paper is to provide a theoretical justification for using local histograms to achieve this goal.}
\end{center}
\end{figure}

Our use of local histograms was motivated by the unique image features found in histology images of teratomas derived from primate embryonic stem cells.  In layman's terms, these tumors begin as masses of undifferentiated cells that are implanted in laboratory animals.  Over time, these tumors grow and their cells differentiate into many various types---bone, cartilage, skin, etc.---until a point at which they are excised, sectioned, stained and viewed under a microscope, resulting in images such as the one in Figure~\ref{figure.Histology}(a).  As such, these images exhibit a wide variety of tissue types, arranged in a seemingly random fashion.  Indeed, to a casual observer such images can appear as a jumbled mess.  In truth however, the arrangement of these tissues is not completely random, and is rather the result of not yet well-understood biological mechanisms.  Drs.~Castro and Ozolek believe that by looking at many such images---many sections of many teratomas---they can gain greater insight into these mechanisms.  Here, spatial context is crucial: one must identify which particular tissue is present at any given point in order to estimate the total amount of each type, as well as the degree to which any given type is adjacent to other types.

In light of these facts, we seek an algorithm which assigns a tissue label to each pixel, thereby segmenting (delineating) and classifying (identifying) the image at the same time. Indeed, such an algorithm would be useful in a broad class of digital pathology applications beyond the teratoma problem~\cite{Bhagavatula10Nov}.  While designing such an algorithm, we must keep in mind that often no single pixel contains enough information to uniquely determine a label.  Rather, the decisions will be made based on features computed over some fixed neighborhood of every given pixel location.   To determine which specific features to use, it helps to have a closer look at each individual tissue.  For example, for the $1200\times1200$ image given in Figure~\ref{figure.Histology}(a) and thumbnailed in Figure~\ref{figure.histograms}(a), we zoom in on three tissue types---cartilage, connective tissue and pseudovascular tissue---resulting in the $128\times128$ subimages given in Figure~\ref{figure.histograms}(b), (c) and (d), respectively.  Each of these three tissue types exhibits a unique aperiodic texture.  For instance, the cartilage texture can be regarded as a light purple field speckled with darker reddish-purple blobs; each blob represents an individual cell's nucleus.  Meanwhile, connective tissue appears as dark purple blobs over a light pink field; pseudovascular tissue is similar to connective tissue, but contains additional reddish-pink structures.  In particular, these three textures exhibit distinct distributions of color, a fact which can quantitatively be confirmed by computing the two-dimensional histograms of their red-blue (RB) pixel value pairs, as depicted in Figure~\ref{figure.histograms}(f), (g) and (h).  
\begin{figure}[t]
\begin{center}
\subfigure[Histology image]
{\includegraphics[width=.24\textwidth]{1a.png}}
\hfill
\subfigure[Cartilage]
{\includegraphics[width=.24\textwidth]{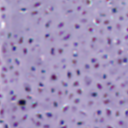}}
\hfill
\subfigure[Connective tissue]
{\includegraphics[width=.24\textwidth]{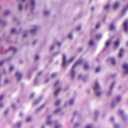}}
\hfill
\subfigure[Pseudovascular tissue]
{\includegraphics[width=.24\textwidth]{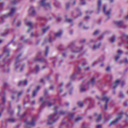}}
\\
\subfigure[The RB histogram of (a).]
{\includegraphics[width=.24\textwidth]{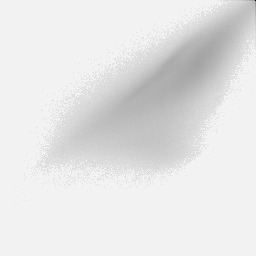}}
\hfill
\subfigure[The RB histogram of (b).]
{\includegraphics[width=.24\textwidth]{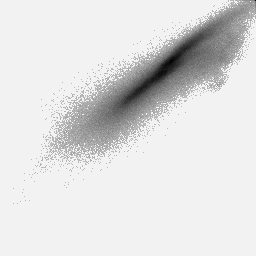}}
\hfill
\subfigure[The RB histogram of (c).]
{\includegraphics[width=.24\textwidth]{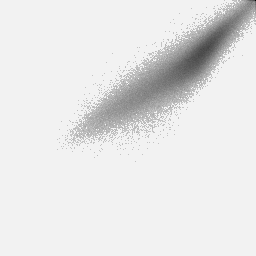}}
\hfill
\subfigure[The RB histogram of (d).]
{\includegraphics[width=.24\textwidth]{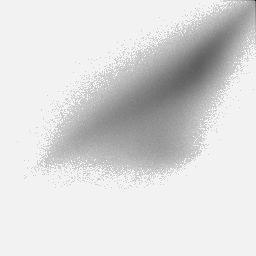}}
\end{center}
\caption{\label{figure.histograms} A $1200\times1200$ histology image exhibiting multiple tissue types (a) and the $256\times256$ histogram of its red-blue (RB) pixel values (e).  As is common with H\&E staining, the tissues in (a) are purple-pink and so we ignore the green component of these red-green-blue (RGB) images when computing (e).  This histogram is viewed from above, with red and blue ranging from $0$ to $255$ on the horizontal and vertical axes, respectively; here the height of the histogram is proportional to darkness for the sake of readability.  In (b), (c) and (d) we zoom in on three $128\times128$ patches extracted from (a), each of which exhibit a single tissue type, namely cartilage, connective tissue and pseudovascular tissue, respectively.  Each of these three tissue types has a distinct distribution of pixel values, as evidenced by their corresponding RB histograms (f), (g) and (h).  These $256\times256$ histograms are similar to (e), but are only computed over those points of a given type according to the ground truth labels in Figure~\ref{figure.Histology}(b).  In particular, the histogram (f) of cartilage (b) is computed over all points labeled in black in Figure~\ref{figure.Histology}(b).  We see that cartilage is darker, on average, than connective: (f) is distributed more towards the lower left-hand side than (g) is.  Moreover, pseudovascular is similar to connective, but possesses additional reddish-pink structures, as evidenced by the subdiagonal blob found in (h), but not (g).  As such, it is plausible that local histograms can serve as discriminating features in segmentation-and-classification algorithms.}
\end{figure}

As certain tissues can be distinguished from others based solely on the distributions of their pixel values, we propose to use histograms as image features in a segmentation-and-classification scheme.  These histograms must be computed locally---over a fixed neighborhood of every pixel location---since global histograms, such as the one depicted in Figure~\ref{figure.histograms}(e) derived from Figure~\ref{figure.histograms}(a), destroy spatial context by mixing all of the individual distributions together.  A similar issue arises in time-frequency analysis: spectrograms preserve spatial context while Fourier transforms do not.  Indeed, local histograms are philosophically similar to spectrograms: in a neighborhood of a given point, the local histogram transform estimates the \textit{frequency of occurrence} of a given value while the spectrogram estimates frequency in the traditional sense.

The purpose of this paper is to provide a mathematically rigorous justification for the use of local histograms in this fashion.  To be precise, we regard our images as functions from a finite abelian group $\calX$ of pixel locations into a second finite abelian group $\calY$ of pixel values.  That is, our images $ f$ are members of the set $\ell(\calX,\calY):=\set{ f:\calX\rightarrow\calY}$.  For example, the $1200\times1200$, $8$-bit red-green-blue (RGB) image given in Figure~\ref{figure.Histology}(a) has $\calX=\bbZ_{1200}^2=\bbZ_{1200}\times\bbZ_{1200}$ and $\calY=\bbZ_{256}^3$, where $\bbZ_N$ denotes the cyclic group of integers modulo $N$.  For purple-pink H\&E-stained images, we often omit the green channel for the sake of computational efficiency, at which point $\calY$ becomes $\bbZ_{256}^2$.  The local histograms of an image $ f$ are defined in terms of a \textit{weighting function}, that is, a nonnegatively-valued $ w\in\ell(\calX,\bbR)$ whose values sum to one.  Specifically, the \textit{local histogram transform} of $ f$ with respect to $ w$ is the function $\LH_ w  f:\calX\times\calY\rightarrow\bbR$,
\begin{equation}
\label{equation.definition of local histogram}
(\LH_ w f)(x,y)
:=\sum_{x'\in\calX} w(x')\delta_y( f(x+x')),
\end{equation}
where $\delta_y(f(x+x'))=1$ if $f(x+x')=y$ and is otherwise zero.  For any fixed $x\in\calX$, the corresponding cross-section of this function, namely $(\LH_ w  f)(x,\cdot):\calY\rightarrow\bbR$, counts the number of instances at which $ f$ obtains a given value $y$ in a $ w$-neighborhood of $x$.

In this paper, we show that local histogram transforms~\eqref{equation.definition of local histogram} are well-suited to the analysis of a particular class of textures.  In short, we want a rigorous explanation of the following hypothesis: say for the sake of argument that 80\% of the cartilage texture in Figure~\ref{figure.histograms}(b) consists of ``background" light purple pixels while the remaining 20\% of pixels lie in a ``foreground" of darker-reddish purple blobs; we then expect a local histogram computed over a portion of cartilage to be a mixture---convex combination---of $0.8$ of the background pixels' distribution with $0.2$ of the foreground pixels' distribution.   Other tissues arise from other distinct decompositions.  For example, looking at the pseudovascular tissue of Figure~\ref{figure.histograms}(d), we might guess it to be $0.5$ light pink, $0.25$ dark purple and $0.25$ reddish-pink.  We rigorously show that such decompositions of local histograms indeed exist for textures arising from a certain class of probabilistic image models; our long-term goal is to exploit this fact in a segmentation-and-classification algorithm.

To see how to formalize these ideas, it helps to consider a toy example: imagine that at any given pixel location, a coin is flipped, with ``heads" resulting in a pink pixel value, and ``tails" resulting in a purple one.  One expects that, on average, the local histogram at any point will consist of two peaks: one in the pink portion of $\calY$, and one in the purple.  Such an image can be regarded as the result of \textit{occluding} a solid purple image $ f_0$ with a solid pink one $ f_1$: at each pixel, the flip of a coin determines whether $ f_1$ lies on top of $ f_0$ at that point, or vice versa.  More generally, the \textit{occlusion} of a set of $N$ images $\set{ f_n}_{n=0}^{N-1}$ in $\ell(\calX,\calY)$ with respect to a given \textit{label function} $\varphi\in\ell(\calX,\bbZ_N)$ is:
\begin{equation}
\label{equation.definition of occlusion}
\bigparen{\occ{\varphi}\set{ f_n}_{n=0}^{N-1}}(x)
:= f_{\varphi(x)}(x).
\end{equation}
That is, at any pixel location $x$, the label $\varphi(x)$ determines which of the potential pixel values $\set{ f_n(x)}_{n=0}^{N-1}$ actually appears in the composite image $\occ{\varphi}\set{ f_n}_{n=0}^{N-1}$ at that point.

The main results of this paper are concerned with when the local histograms~\eqref{equation.definition of local histogram} of a composite image~\eqref{equation.definition of occlusion} are related to the local histograms of the individual $ f_n$'s.  Though it is unrealistic to expect a clean relation for any fixed $\varphi$, we can show that these quantities are indeed closely related, provided one averages over all possible label functions $\varphi$.  Indeed, denoting the probability of getting ``heads" in the above toy example as $\rho\in[0,1]$, we would expect the volumes of the pink and purple peaks of the composite image's local histograms to be $\rho$ and $1-\rho$, respectively.  That is, $\LH_ w\occ{\varphi}\set{ f_0, f_1}$ should be $(1-\rho)\LH_ w  f_0+\rho\LH_ w  f_1$, on average.  We generalize this idea so as to permit more realistic textures with more colors and with spatially-correlated pixels. 

To be precise, fix a set of source images $\set{ f_n}_{n=0}^{N-1}$ and consider the set $\set{\occ{\varphi}\set{ f_n}_{n=0}^{N-1}}_{\varphi\in\ell(\calX,\bbZ_N)}$ of all possible composite images~\eqref{equation.definition of occlusion} obtained by letting $\varphi$ be any one of the $N^{\abs{\calX}}$ elements of $\ell(\calX,\bbZ_N)$, where $\abs{\calX}$ denotes the cardinality of $\calX$.  We refer to a random method for choosing one of these composites as an \textit{occlusion model} $\Phi$.  Formally speaking, $\Phi$ is a random variable version of $\varphi$, meaning there exists a probability density function $\rmP_\Phi:\ell(\calX,\bbZ_N)\rightarrow[0,1]$ such that $\sum_{\varphi\in\ell(\calX,\bbZ_N)}\rmP_\Phi(\varphi)=1$.  For example, imagine three $128\times128$ images $ f_0$, $ f_1$ and $ f_2$ which exhibit a nearly constant shade of pink, purple and red, respectively.  Given any label function $\varphi:\bbZ_{128}^2\rightarrow\bbZ_3$ we can produce a corresponding $128\times 128$ composite image $\occ{\varphi}\set{ f_0, f_1, f_2}$ whose pixels are some mixture of pink, purple and red.  For some choices of $\varphi$ the resulting composites will look like the pseudovascular tissue texture given in Figure~\ref{figure.histograms}(d).  However, even in this small example, there are an enormous number of such possible composites---one for each of the \smash{$3^{128^2}$} possibilities for $\varphi$---and only a few of these will look like pseudovascular tissue; most will appear as pink-purple-red static.  The role of the occlusion model $\Phi$ is to assign a probability to each of these possible $\varphi$'s in a manner that emphasizes those textures one expects to appear in a given tissue while de-emphasizing the rest.

In this paper, we provide a sufficient hypothesis on the occlusion model $\Phi$ so as to ensure that the local histograms~\eqref{equation.definition of local histogram} of a composite image~\eqref{equation.definition of occlusion} can, on average with respect to $\rmP_\Phi$, be decomposed in terms of the local histograms of the individual images.  In particular, we focus on the special case where the occlusion model $\Phi$ is \textit{flat}, meaning that on average, the probability that $\Phi$ chooses label $n$ at a given pixel location $x$ is equal to the probability of choosing $n$ at any other $x'$; formally, $\Phi$ is \textit{flat} if there exists scalars $\set{\lambda_n}_{n=0}^{N-1}$ such that:
\begin{equation}
\label{equation.definition of flatness}
\sum_{\substack{\varphi\in\ell(\calX,\bbZ_N)\\\varphi(x)=n}}\rmP_\Phi(\varphi)=\lambda_n,\quad \forall x\in\calX.
\end{equation}
That is, $\Phi$ is flat if the marginal distributions obtained by fixing any given $x\in\calX$ are identical.  Note that for any fixed $x\in\calX$, summing~\eqref{equation.definition of flatness} over all $n$ yields that $\sum_{n=1}^{N}\lambda_n=1$.  Indeed, at any given pixel location $x$, the value $\lambda_n$ represents the probability that the random label function $\Phi$ will have label $n$ at that $x$.  In our toy example where the values of $\varphi$ are determined by $\abs{\calX}$ spatially-independent coin flips, the probability of getting any particular $\varphi\in\ell(\calX,\bbZ_2)$ is $\rmP_\Phi(\varphi)=\rho^{\abs{\varphi^{-1}\set{1}}}(1-\rho)^{\abs{\calX}-\abs{\varphi^{-1}\set{1}}}$; substituting this expression into \eqref{equation.definition of flatness}, the binomial theorem implies that this model is indeed flat with $\lambda_0=1-\rho$ and $\lambda_1=\rho$.  Note that, if $\rho>\frac12$, the resulting random image $\occ{\Phi}\set{ f_0, f_1}$ will be more pink than purple; flatness does not mean that each label is equally likely, but rather that the chance of being pink at any given pixel location is the same as at any other location.  These concepts in hand, we present one of our main results, which formally claims that, on average, the local histograms of composite images produced from flat occlusion models are but mixtures of the local histograms of the source images:
\begin{thm}
\label{theorem.flatness}
If $\Phi$ is flat as in~\eqref{equation.definition of flatness}, then the expected value of the local histogram transform~\eqref{equation.definition of local histogram} of a composite image~\eqref{equation.definition of occlusion} is a convex combination of the local histograms of each individual image:
\begin{equation}
\label{equation.theorem on flatness 1}
\sum_{\varphi\in\ell(\calX,\bbZ_N)}\rmP_\Phi(\varphi)(\LH_ w\occ{\varphi}\set{ f_n}_{n=0}^{N-1})(x,y)
=\sum_{n=0}^{N-1}\lambda_n(\LH_ w f_n)(x,y).
\end{equation}
\end{thm}
\noindent From the point of view of our motivating application, the significance of Theorem~\ref{theorem.flatness} is that it gives credence to a certain type of segmentation-and-classification algorithm.  To be precise, given a set of training images which are manually segmented and labeled by medical experts, we, for any given tissue type, can compute local histograms over regions which are labeled as that type.  In light of Theorem~\ref{theorem.flatness}, it is reasonable to \textit{demix}---decompose into convex combinations---the local histograms of that type into a set of more basic distributions.  For example, we expect that the local histograms of pseudovascular tissue (Figure~\ref{figure.histograms}(d)) can be demixed into three simpler distributions---one pink, another purple and a third reddish-pink---while those of connective tissue (Figure~\ref{figure.histograms}(c)) are mixtures of only the first two.  Once sparse demixings of each tissue type have been found, we then use them to segment and classify: given a new image,  we assign a label at any given point by determining which particular set of learned distributions its local histogram is most consistent with.  

The remainder of our main results are in support of this interpretation of Theorem~\ref{theorem.flatness}.  Specifically, the next section contains several basic results on local histograms.  In Section~\ref{Section.Occlusion}, we prove Theorem~\ref{theorem.flatness} and also a generalization of it---Theorem~\ref{theorem.expected value of lh of occ}---to the non-flat case.  In Section~\ref{Section.Flatness}, we provide various methods---Theorems~\ref{theorem.PJ},~\ref{theorem.star product} and~\ref{theorem.overlay}---for constructing flat $\Phi$'s, and some of these produce textures that resemble those found in digital microscope images of histological tissues.  The final section discusses a preliminary segmentation-and-classification algorithm inspired by Theorem~\ref{theorem.flatness} in which local histograms are decomposed using principal component analysis (PCA).

Both local histograms and probabilistic image occlusion models have long been subjects of interest.  Theorem~\ref{theorem.PropertiesOfLHs} below details how local histograms can be computed as systems of convolutions; a similar result is given in~\cite{Kober93}, and both \cite{Kober93} and \cite{Szoplik96} discuss how such a computation can be implemented in optical hardware.  Recently, local histograms have been used in an active contour-based segmentation scheme~\cite{Ni09}; this algorithm partitions an image into two smoothly bounded regions whose pixel values are maximally separated with respect to the Wasserstein (earth mover's) distance.  Local histograms have also recently been used as smoothing filters~\cite{Kass10}.  Though the work we present here focuses exclusively on local histograms of the pixel values themselves, an alternative approach is to first pass the image through a filter bank and then compute histograms of the resulting values~\cite{Hadjidemetriou04,Liu06}.  Local histograms, like time-frequency transforms, preserve global spatial context while obscuring all local spatial context, and as such they are well-suited to the processing of \textit{locally orderless} images~\cite{Ginneken00,Koenderink99,Koenderink00}.  We use local histograms to analyze a class of textures generated by a certain probabilistic occlusion model; this model, like the~\textit{dead leaves model}~\cite{Bordenave06,Lee01,Mumford00}, generates these textures via a sequential superposition of random sets.  Our contribution to this body of literature is a formalism that unifies the theory of local histograms with that of occlusion models and permits us to rigorously prove that local histograms are indeed a useful transform for the analysis of a particular class of textures.

%%%%%%%%%%%%%%%%%%%%%%%%%%%%%%%%%%%%%%%%%%%%%%%%%%%%%%%%%%%%%%%%
%%%%%%%%%%%%%%%%%%%%%%%%%%%%%%%%%%%%%%%%%%%%%%%%%%%%%%%%%%%%%%%%
% Section 2: Local Histograms
%%%%%%%%%%%%%%%%%%%%%%%%%%%%%%%%%%%%%%%%%%%%%%%%%%%%%%%%%%%%%%%%
%%%%%%%%%%%%%%%%%%%%%%%%%%%%%%%%%%%%%%%%%%%%%%%%%%%%%%%%%%%%%%%%

\section{Local histograms}
\label{Section.Local Histograms}
In this section, we discuss an efficient means of computing local histograms~\eqref{equation.definition of local histogram} and discuss several of their basic properties.  Computing local histograms can be time consuming, especially as $\calX$ and $\calY$ become large.  In particular, for a general window $ w$, a direct computation of \eqref{equation.definition of local histogram} requires $\mathcal{O}(\abs{\calX}^2\abs{\calY})$ operations: $\mathcal{O}(\abs{\calX})$ operations for each $x\in\calX$ and $y\in\calY$.  A more efficient method is given in Theorem~\ref{theorem.PropertiesOfLHs} below: \eqref{equation.definition of local histogram} can be computed as a system of $\abs{\calY}$ convolutions over $\calX$, which only requires $\mathcal{O}(\abs{\calX}\abs{\calY}\log\abs{\calX})$ operations if discrete Fourier transforms are used.  In particular, we filter the \textit{characteristic function} of the graph of $ f$, namely $1_ f:\calX\times\calY\rightarrow\bbR$,
\begin{equation}
\label{equation.CharFunction}
1_ f(x,y)
:=1_{ f^{-1}{\left\{y\right\}}}(x)
=\delta_y( f(x))
=\left\{\begin{array}{ll}1,& f(x)=y,\\0,& f(x)\neq y,\end{array}\right.
\end{equation}
with the \textit{reversal} of $ w\in\ell(\calX,\bbR)$, namely $\tilde w(x):= w(-x)$.  This method for computing local histograms is illustrated in Figure~\ref{figure.graph}.
\begin{figure}
\begin{center}
\includegraphics{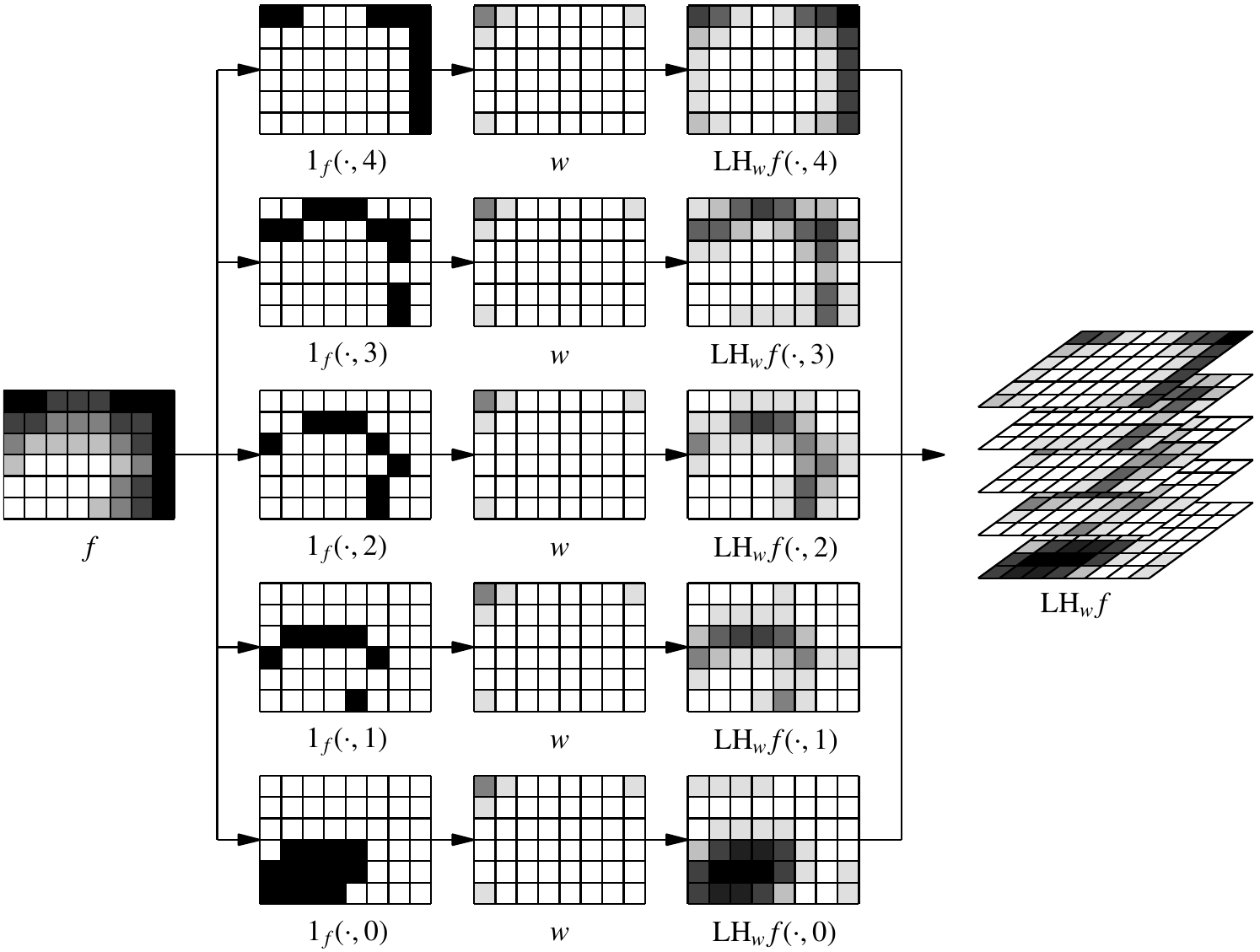}
\caption{\label{figure.graph}
An example of how to compute local histograms using Theorem~\ref{theorem.PropertiesOfLHs}(a).  For the sake of readability, larger numerical values are represented by darker shades throughout.  The source image (far left) $ f$ is $6\times 8$ and has grayscale values ranging from $0$ to $4$.  That, is $ f\in\ell(\calX,\calY)$ where $\calX=\bbZ_6\times\bbZ_8$ and $\calY=\bbZ_5$.  Its characteristic function~\eqref{equation.CharFunction} is a $\set{0,1}$-valued $6\times8\times5$ data cube whose cross-sections (left column) indicate those locations at which $ f$ attains any given value.  By Theorem~\ref{theorem.PropertiesOfLHs}(a), the $6\times 8\times 5$ data cube that contains the local histograms of $ f$ (far right) may be computed one level at a time (right column) by filtering these binary-valued cross-sections with a real-scalar-valued weighting function (middle column).  In this simple example, the weighting function is $w=\tfrac{1}{2}\delta_{0,0}+\tfrac{1}{8}(\delta_{-1,0}+\delta_{1,0}+\delta_{0,-1}+\delta_{0,1})$, where the origin lies in the upper left-hand corner of the grid.}
\end{center}
\end{figure}
Alternatively, \eqref{equation.definition of local histogram} can be computed as a single convolution over $\calX\times\calY$; here, the tensor product of $ w\in\ell(\calX,\bbR)$ with $\omega\in\ell(\calY,\bbR)$ is defined as $ w\otimes\omega\in\ell(\calX\times\calY,\bbR)$, $( w\otimes\omega)(x,y):= w(x)\omega(y)$.
\begin{thm}
\label{theorem.PropertiesOfLHs}
For any $ w\in\ell(\calX,\bbR)$, $\omega\in\ell(\calY,\bbR)$, $ f\in\ell(\calX,\calY)$, $x\in\calX$, and $y\in\calY$:
\begin{enumerate}
\alphi
\item
Local histograms \eqref{equation.definition of local histogram} can be evaluated as a system of $\abs{\calY}$ convolutions over $\calX$: $(\LH_ w  f)(x,y)=(\tilde{ w}\ast1_{ f^{-1}\{y\}})(x)$.
\item
Alternatively, \eqref{equation.definition of local histogram} may be computed as a single convolution over $\calX\times\calY$: $\conv{(\delta_0\otimes\omega)}{\LH_ w f}=\conv{(\tilde{ w}\otimes\omega)}{1_ f}$.

In particular, taking $\omega=\delta_0$ gives $\LH_ w f=\conv{(\tilde{ w}\otimes\delta_0)}{1_ f}$.
\end{enumerate}
\end{thm}
\begin{proof}
For (a), replacing $x'$ with $-x'$, and substituting the relation $\delta_{y}( f(x-x'))=1_{ f^{-1}\{y\}}(x-x')$ into \eqref{equation.definition of local histogram} yields:
\begin{equation}
(\LH_ w f)(x,y)
=\sum_{x'\in\calX} w(x')\delta_y( f(x+x'))
=\sum_{x'\in\calX} w(-x')1_{ f^{-1}\{y\}}(x-x')
=\sum_{x'\in\calX}\tilde{ w}(x')1_{ f^{-1}\{y\}}(x-x')
=(\tilde{ w}\ast1_{ f^{-1}\{y\}})(x).
\label{equation.2DConvResult}
\end{equation}
For (b), the definition of $\delta_0$ gives:
\begin{equation}
[\conv{(\delta_0\otimes\omega)}{\LH_ w f}](x,y)
=\sum_{(x',y')\in\calX\times\calY}(\delta_0\otimes\omega)(x',y')(\LH_ w f)(x-x',y-y')
=\sum_{y'\in\calY}\omega(y')(\LH_ w f)(x,y-y').
\label{equation.3DConvResult}
\end{equation}
Substituting \eqref{equation.2DConvResult} into \eqref{equation.3DConvResult} and using \eqref{equation.CharFunction}, gives our result:
\begin{align*}
\nonumber
[\conv{(\delta_0\otimes\omega)}{\LH_ w f}](x,y)
&=\sum_{y'\in\calY}\omega(y')(\conv{\tilde{ w}}{1_{ f^{-1}\{y-y'\}}})(x)\\
\nonumber
&=\sum_{y'\in\calY}\omega(y')\sum_{x'\in\calX}\tilde{ w}(x')1_{ f^{-1}\{y-y'\}}(x-x')\\
\nonumber
&=\sum_{(x',y')\in\calX\times\calY}(\tilde{ w}\otimes\omega)(x',y')1_ f(x-x',y-y')\\
&=[\conv{(\tilde{ w}\otimes\omega)}{1_ f}](x,y).\qedhere
\end{align*}
\end{proof}
\noindent
The next result summarizes several other basic properties of local histograms, the proofs of which are given in~\cite{Bhagavatula10Nov,Massar10Dec}.
\begin{prop}
\label{proposition.LHresults}
For any $ w\in\ell(\calX,\bbR)$ and $ f\in\ell(\calX,\calY)$:
\begin{enumerate}
\alphi
\item
The levels of a local histogram transform sum to $1$: for any $x\in\calX$, $\sum_{y\in\calY}(\LH_ w f)(x,y)=1$.
\item
Local histograms commute with spatial translation $\rmT^{x}$: for any $x\in\calX$, $\LH_ w \rmT^{x}=\rmT^{(x,0)}\LH_ w$ .
\item
Adding constants to images shifts their local histograms along $\calY$: for any $y\in\calY$, $\LH_ w( f+y)=\rmT^{(0,y)}\LH_ w f$ .
\item
Quantizing an image will bin its local histograms: for any $q\in\ell(\calY,\calY')$,
\begin{equation*}
[\LH_ w(q\circ f)](x,y')=\sum_{\substack{y\in\calY\\q(y)=y'}}(\LH_ w f)(x,y).
\end{equation*}
\end{enumerate}
\end{prop}
This basic understanding of local histograms in hand, we turn to the theory of applying them to textures generated by the probabilistic image occlusion models discussed in the introduction.

%%%%%%%%%%%%%%%%%%%%%%%%%%%%%%%%%%%%%%%%%%%%%%%%%%%%%%%%%%%%%%%%
%%%%%%%%%%%%%%%%%%%%%%%%%%%%%%%%%%%%%%%%%%%%%%%%%%%%%%%%%%%%%%%%
% Section 3: A probabilistic image occlusion model
%%%%%%%%%%%%%%%%%%%%%%%%%%%%%%%%%%%%%%%%%%%%%%%%%%%%%%%%%%%%%%%%
%%%%%%%%%%%%%%%%%%%%%%%%%%%%%%%%%%%%%%%%%%%%%%%%%%%%%%%%%%%%%%%%

\section{Local histograms of randomly-generated textures}
\label{Section.Occlusion}

In this section, we rigorously confirm our intuition regarding local histograms of textures generated via random occlusions: if a texture, such as that found in the pseudovascular tissue of Figure~\ref{figure.histograms}(d), is some sufficiently-spatially-random combination of $50\%$ pink pixels, $25\%$ purple pixels and $25\%$ red pixels, then its local histograms should, on average, be a mixture of three simpler distributions, namely a convex combination of $0.5$ of a purely pink distribution with $0.25$ purely purple and red ones.

To do this, fix any set of $N$ source images $\set{ f_n}_{n=0}^{N-1}$ and let $\Phi$ be any occlusion model as defined in the introduction.  That is, let $\Phi$ be a random variable version of a label function $\varphi:\calX\rightarrow\bbZ_N$, as defined by a probability density function $\rmP_\Phi:\ell(\calX,\bbZ_N)\rightarrow[0,1]$ where $\sum_{\varphi\in\ell(\calX,\bbZ_N)}\rmP_\Phi(\varphi)=1$.  In the results that follow, a useful quantity to consider is the expected value---with respect to $\rmP_\Phi$---of the characteristic function $1_\varphi$ obtained by letting $ f=\varphi$ in~\eqref{equation.CharFunction}:
\begin{equation}
\label{equation.ChiBar}
\overline1_\Phi(x,n)
:=\sum_{\varphi\in\ell(\calX,\bbZ_N)}\rmP_\Phi(\varphi)1_\varphi(x,n)
=\sum_{\substack{\varphi\in\ell(\calX,\bbZ_N)\\\varphi(x)=n}}\rmP_\Phi(\varphi).
\end{equation}
Essentially, $\overline1_\Phi(x,n)$ is the probability that a random label function $\varphi$ generated by the occlusion model $\Phi$ will assign label $n$ to pixel location $x$.  When compared with the definition of flatness~\eqref{equation.definition of flatness}, we see that $\Phi$ is flat if and only if there exist scalars $\set{\lambda_n}_{n=0}^{N-1}$ such that $\overline1_\Phi(x,n)=\lambda_n$ for all $x\in\calX$ and $n\in\bbZ_N$.  That is, $\Phi$ is flat if and only if $\overline1_\Phi(x,n)$ is constant with respect to pixel location $x$.  Having this concept, we present one of our main results:
\begin{thm}
\label{theorem.expected value of lh of occ}
For any sequence of images $\{ f_n\}_{n=0}^{N-1}\in\ell(\calX,\calY)$, weighting function $ w$ and any $N$-image occlusion model $\Phi$, the expected value of the local histogram~\eqref{equation.definition of local histogram} of the composite image~\eqref{equation.definition of occlusion} with respect to $ w$ is:
\begin{equation}
\label{equation.OcclusionResult}
\rmE_\Phi(\LH_ w\occ{\Phi}{\{ f_n\}_{n=0}^{N-1}})(x,y)
=\sum_{n=0}^{N-1}\overline1_\Phi(x,n)(\LH_ w f_n)(x,y)+\varepsilon,
\end{equation}
where the error term $\varepsilon$ is bounded by $\abs{\varepsilon}\leq\displaystyle\sum_{n=0}^{N-1}\sum_{x'\in\calX} w(x')\abs{\overline1_\Phi(x+x',n)-\overline1_\Phi(x,n)}$.  Moreover, 
\begin{equation}
\label{equation.expected value of local histograms of occluded images}
\sum_{n=0}^{N-1}\overline1_\Phi(x,n)
=1,
\end{equation}
and so \eqref{equation.OcclusionResult} states that, on average, the local histograms of the composite image $\occ{\varphi}\{ f_n\}_{n=0}^{N-1}$ can be approximated by convex combinations of local histograms of each individual image $ f_n$.
\end{thm}

\begin{proof}
The expected value of the local histogram \eqref{equation.definition of local histogram} of a composite image \eqref{equation.definition of occlusion} is:
\begin{equation}
\label{equation.proof of expected value of lh of occ 1}
\rmE_\Phi(\LH_ w\occ{\Phi}\{ f_n\}_{n=0}^{N-1})(x,y)
=\sum_{\varphi\in\ell(\calX,\bbZ_N)}\rmP_\Phi(\varphi)\sum_{x'\in\calX} w(x')\delta_y((\occ{\varphi}\{ f_n\}_{n=0}^{N-1})(x+x')).
\end{equation}
For any fixed $\varphi$, $x$, and $x'$, we have $\varphi(x+x')=n$ for exactly one $n$.  For any fixed $x$, $x'$ and $y$, we can therefore split a sum of $1_\varphi(x+x',n)\delta_{y}( f_n(x+x'))$ over all $n$ into one summand where $n=\varphi(x+x')$ and the remaining $N-1$ summands for which $n\neq\varphi(x+x')$:
\begin{equation}
\label{equation.proof of expected value of lh of occ 2}
\sum_{n=0}^{N-1}1_\varphi(x+x',n)\delta_y( f_n(x+x'))
=(1)\delta_y( f_{\varphi(x+x')}(x+x'))+\sum_{n\neq\varphi(x+x')}(0)\delta_y( f_n(x+x'))
=\delta_y((\occ{\varphi}\{ f_n\}_{n=0}^{N-1})(x+x')),
\end{equation}
where the final equality follows immediately from \eqref{equation.definition of occlusion}.
Substituting \eqref{equation.proof of expected value of lh of occ 2} into \eqref{equation.proof of expected value of lh of occ 1} and using \eqref{equation.ChiBar} yields:
\begin{align}
\nonumber
\rmE_\Phi(\LH_ w\occ{\Phi}\{ f_n\}_{n=0}^{N-1})(x,y)
&=\sum_{\varphi\in\ell(\calX,\bbZ_N)}\rmP_\Phi(\varphi)\sum_{x'\in\calX} w(x')\biggparen{\,\sum_{n=0}^{N-1}1_\varphi(x+x',n)\delta_y( f_n(x+x'))}\\
\nonumber
&=\sum_{n=0}^{N-1}\sum_{x'\in\calX} w(x')\delta_y( f_n(x+x'))\sum_{\varphi\in\ell(\calX,\bbZ_N)}\rmP_\Phi(\varphi)1_\varphi(x+x',n)\\
\label{equation.proof of expected value of lh of occ 3}
&=\sum_{n=0}^{N-1}\sum_{x'\in\calX} w(x')\delta_y( f_n(x+x'))\overline1_\Phi(x+x',n).
\end{align}
Rewriting~\eqref{equation.proof of expected value of lh of occ 3} in terms of $\displaystyle\varepsilon:=\sum_{n=0}^{N-1}\sum_{x'\in\calX} w(x')\delta_y( f_n(x+x'))[\overline1_\Phi(x+x',n)-\overline1_\Phi(x,n)]$ gives our first claim~\eqref{equation.OcclusionResult}:
\begin{align*}
\rmE_\Phi(\LH_ w\occ{\Phi}\{ f_n\}_{n=0}^{N-1})(x,y)
&=\sum_{n=0}^{N-1}\sum_{x'\in\calX} w(x')\delta_y( f_n(x+x'))\overline1_\Phi(x,n)+\varepsilon\\
&=\sum_{n=0}^{N-1}\overline1_\Phi(x,n)\sum_{x'\in\calX} w(x')\delta_y( f_n(x+x'))+\varepsilon\\
&=\sum_{n=0}^{N-1}\overline1_\Phi(x,n)(\LH_ w f_n)(x,y)+\varepsilon.
\end{align*}
For the second claim, we bound $\varepsilon$ using the triangle inequality and the fact that $\abs{\delta_y( f_n(x+x'))}\leq1$:
\begin{equation*}
\abs{\varepsilon}
=\Biggabs{\sum_{n=0}^{N-1}\sum_{x'\in\calX} w(x')\delta_y( f_n(x+x'))[\overline1_\Phi(x+x',n)-\overline1_\Phi(x,n)]}
\leq\sum_{n=0}^{N-1}\sum_{x'\in\calX} w(x')\abs{\overline1_\Phi(x+x',n)-\overline1_\Phi(x,n)}.
\end{equation*} 
Finally, to prove our third claim~\eqref{equation.expected value of local histograms of occluded images}, note that for any fixed $x\in\calX$, \eqref{equation.ChiBar} gives:
\begin{equation}
\label{equation.proof of expected value of lh of occ 4}
\sum_{n=0}^{N-1}\overline1_\Phi(x,n)
=\sum_{n=0}^{N-1}\sum_{\varphi\in\ell(\calX,\bbZ_{N})}\rmP_\Phi(\varphi)1_\varphi(x,n)
=\sum_{\varphi\in\ell(\calX,\bbZ_{N})}\rmP_\Phi(\varphi)\sum_{n=0}^{N-1}\left\{\begin{array}{ll}1,&\varphi(x)=n,\\0,&\varphi(x)\neq n.\end{array}\right.
\end{equation}
Since as previously noted we have $\varphi(x)=n$ for exactly one $n$, \eqref{equation.proof of expected value of lh of occ 4} becomes: $\displaystyle\sum_{n=0}^{N-1}\overline1_\Phi(x,n)=\sum_{\varphi\in\ell(\calX,\bbZ_{N})}\rmP_\Phi(\varphi)=1$.
\end{proof}
An example illustrating the direct computation of the left-hand side of~\eqref{equation.OcclusionResult} is given in Figure~\ref{figure.ELH}.
\begin{figure}[t]
\begin{center}
\includegraphics{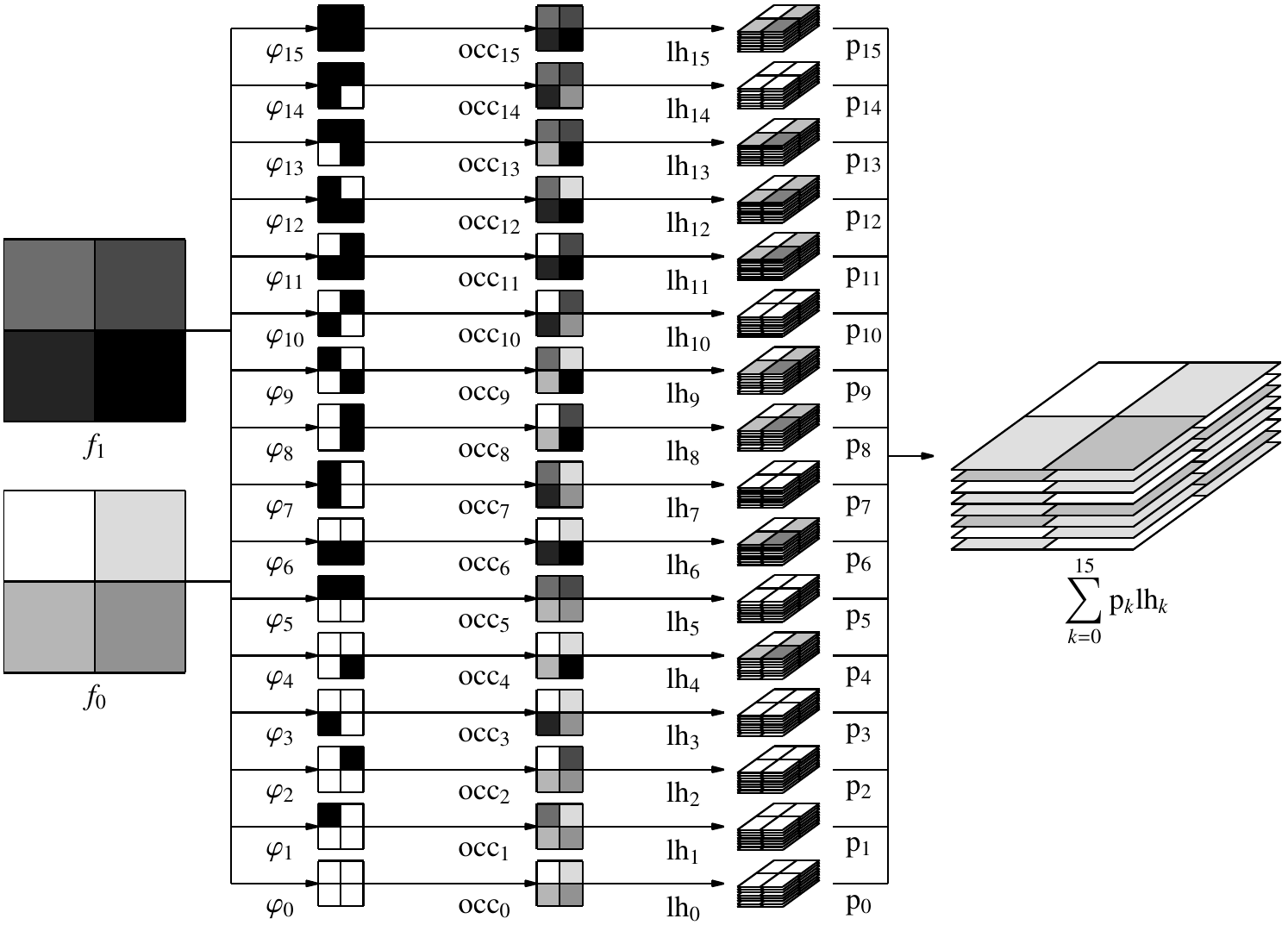}
\caption{
\label{figure.ELH}
An example of how to compute the left-hand side of~\eqref{equation.OcclusionResult} explicitly as a probability-weighted sum.  For the sake of readability, larger numerical values are represented by darker shades throughout.  We consider two $2\times 2$, $3$-bit images, namely $\set{ f_n}_{n=0}^{N-1}$ in $\ell(\calX,\calY)$ where $N=2$, $\calX=\bbZ_2\times\bbZ_2$ and $\calY=\bbZ_8$.  In this particular example, the values of the $ f_n$'s are all distinct, with $ f_0$ assuming values $\set{0,1,2,3}$ and $ f_1$ assuming values $\set{4,5,6,7}$ (far left).  There are \smash{$N^{\abs{\calX}}=2^{2^2}=16$} distinct label functions $\varphi:\bbZ_2\times\bbZ_2\rightarrow\bbZ_2$ (left column) each yielding a composite image $\occ{\varphi}\set{ f_0, f_1}$ (center column); in accordance with~\eqref{equation.definition of occlusion}, we take values from $ f_0$ in places where $\varphi$ is white and values from $ f_1$ where $\varphi$ is black.  Each of these composites has a $2\times 2\times 8$ local histogram transform~\eqref{equation.definition of local histogram} (right column).  Since occlusion~\eqref{equation.definition of occlusion} is nonlinear, there is no clean relationship between the local histograms of any single composite and the local histograms of the source images $ f_0$ and $ f_1$.  Nevertheless, under certain hypotheses, we can say something about the average of these local histograms (far right) with respect to some probability density function $\rmP_\Phi$ on the set $\ell(\bbZ_2\times\bbZ_2,\bbZ_2)$ of all possible $\varphi$'s.  In particular, if the occlusion model $\Phi$ is flat~\eqref{equation.definition of flatness}, meaning in this case that the probability-weighted-sum of all $\varphi$'s is a constant function of $x$, then Theorem~\ref{theorem.flatness} states that this average is a convex combination of the local histograms of $ f_0$ and $ f_1$ as depicted in Figure~\ref{figure.ELH2}.
}
\end{center}
\end{figure}
Note that Theorem~\ref{theorem.expected value of lh of occ} implies that the error term $\varepsilon$ in~\eqref{equation.OcclusionResult} will be small provided the probability $\overline1_\Phi(x,n)$ of assigning label $n$ to $x$ changes little as $x$ varies over regions smaller than the the support of $ w$.  The extreme case of this is when the occlusion model $\Phi$ is flat, meaning $\overline1_\Phi(x,n)$ is constant with respect to $x$.  In this case, $\varepsilon$ vanishes entirely, leading to Theorem~\ref{theorem.flatness} as given in the introduction:
\begin{proof}[Proof of Theorem~\ref{theorem.flatness}]
If $\Phi$ is flat, $\overline1_\Phi(x+x',n)=\overline1_\Phi(x,n)$ for all $x,x'\in\calX$.  The error bound in Theorem~\ref{theorem.expected value of lh of occ} then gives $\varepsilon=0$.  Denoting $\overline1_\Phi(x,n)$ as $\lambda_n$ in~\eqref{equation.OcclusionResult} thus yields our claim.
\end{proof}
That is, when $\Phi$ is flat, \eqref{equation.OcclusionResult} simplifies to~\eqref{equation.theorem on flatness 1}, and so the in-depth computation of Figure~\ref{figure.ELH} can be replaced by the much simpler one depicted in Figure~\ref{figure.ELH2}.
\begin{figure}
\begin{center}
{\includegraphics{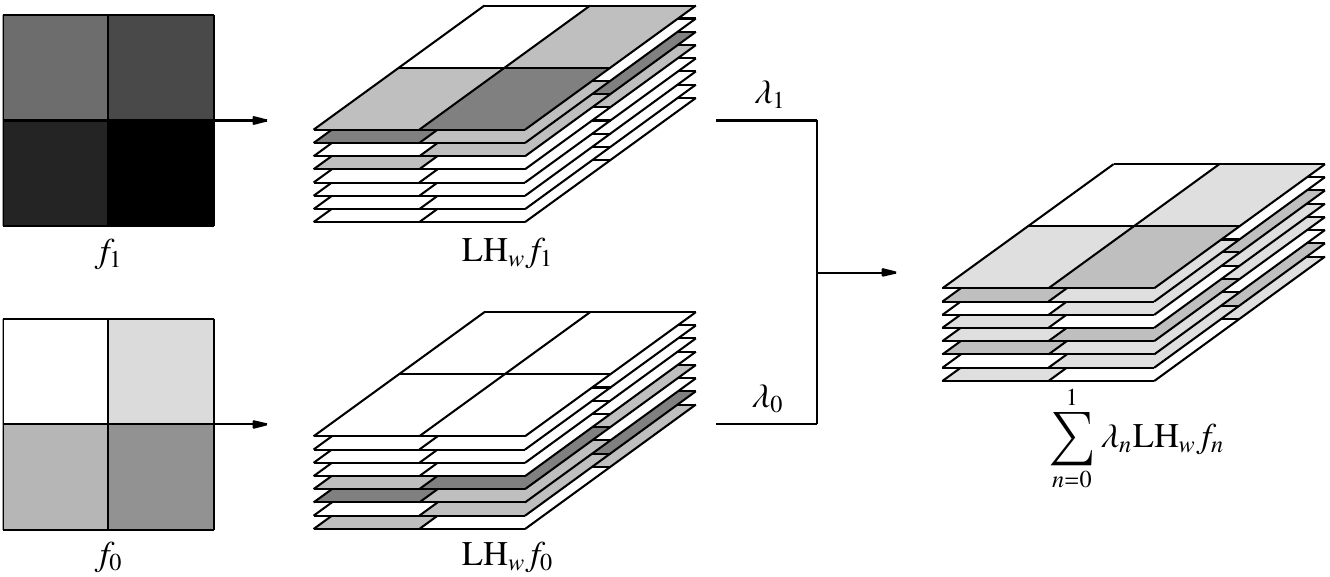}}
\caption{
\label{figure.ELH2}
A continuation of the example of Figure~\ref{figure.ELH}.  When the occlusion model $\Phi$ is flat, Theorem~\ref{theorem.expected value of lh of occ} becomes Theorem~\ref{theorem.flatness}, with \eqref{equation.OcclusionResult} simplifying to~\eqref{equation.theorem on flatness 1}.  Though each of the $16$ distinct composite images shown in Figure~\ref{figure.ELH} has a distinct local histogram transform, the average of these $16$ local histogram transforms with respect to $\rmP_\Phi$ is but a convex combination (right) of the local histograms (center) of the two source images (left).  That is, when $\Phi$ is flat, the average-over-all-composites local histogram computed in Figure~\ref{figure.ELH} is equal to the average-over-all-sources local histogram computed above.}
\end{center}
\end{figure}
Thus, flatness is indeed an important theoretical assumption for the analysis of local histograms of textures generated via random occlusions.  It nevertheless remains to be shown that flatness is also a realistic assumption from the point of view of our motivating application; this is the topic of the next section.

%%%%%%%%%%%%%%%%%%%%%%%%%%%%%%%%%%%%%%%%%%%%%%%%%%%%%%%%%%%%%%%%
%%%%%%%%%%%%%%%%%%%%%%%%%%%%%%%%%%%%%%%%%%%%%%%%%%%%%%%%%%%%%%%%
% Section 4: Flatness/Flat Sets
%%%%%%%%%%%%%%%%%%%%%%%%%%%%%%%%%%%%%%%%%%%%%%%%%%%%%%%%%%%%%%%%
%%%%%%%%%%%%%%%%%%%%%%%%%%%%%%%%%%%%%%%%%%%%%%%%%%%%%%%%%%%%%%%%

\section{Flat occlusion models}
\label{Section.Flatness}

Theorem~\ref{theorem.flatness} gives some insight into the behavior of the local histograms of images generated via random occlusions.  However, this result only holds when $\Phi$ is flat~\eqref{equation.definition of flatness}, namely when its average characteristic function $\overline1_\Phi(x,n)$, as defined in~\eqref{equation.ChiBar}, is constant with respect to pixel location $x$, but is still permitted to vary with label value $n$.  In this section, we demonstrate that flatness is a reasonable assumption.  In particular, we provide a variety of methods for constructing flat occlusion models.  Some of these models produce textures similar to those encountered in digital microscope images of histological tissues.  Our first method involves the translation operator $\rmT^x:\calX\rightarrow\calX$, $\rmT^x\varphi(x'):=\varphi(x'-x)$.  To be precise, we show that an occlusion model $\Phi$ is flat if it is \textit{translation-invariant}, meaning that its probability density function $\rmP_\Phi$ satisfies:
\begin{equation}
\label{equation.definition of translation invariance}
\rmP_\Phi(\rmT^x\varphi)=\rmP_\Phi(\varphi),\quad \forall \varphi\in\ell(\calX,\bbZ_{N}), x\in\calX.
\end{equation}
\begin{thm}
\label{theorem.PJ}
If $\Phi$ is translation-invariant~\eqref{equation.definition of translation invariance}, then $\Phi$ is flat~\eqref{equation.definition of flatness}.
\end{thm}
\begin{proof}
We begin by placing an equivalence relation $\sim$ on $\ell(\calX,\bbZ_N)$, letting $\varphi'\sim\varphi$ when there exists some $x\in\calX$ such that $\varphi'=\rmT^x\varphi$.  Letting $\calR$ denote a set of representatives from the corresponding equivalence classes, we have that for all $\varphi'\in\ell(\calX,\bbZ_N)$, there exists a unique $\varphi\in\calR$ such that $\varphi'=\rmT^x\varphi$.  As such, 
\begin{equation}
\label{equation.proof of translation-invariant implies flat j1}
\overline1_\Phi
=\sum_{\varphi\in\ell(\calX,\bbZ_N)}\rmP_{\Phi}(\varphi)1_{\varphi}
=\sum_{\varphi\in\calR}\sum_{\varphi'\sim\varphi}\rmP_\Phi(\varphi')1_{\varphi'}.
\end{equation}
Now, fix any $\varphi\in\calR$, and consider the subgroup $\calW_\varphi=\set{x\in\calX : \rmT^x\varphi=\varphi}$ of the finite abelian group $\calX$.  Letting $\calX/\calW_\varphi$ denote a fixed set of coset representatives of $\calX$ with respect to $\calW_\varphi$, we claim that $\beta:\calX/\calW_\varphi\rightarrow\{\varphi':\varphi'\sim\varphi\}$, $\beta(x):=\rmT^x\varphi$ is a bijection.

Indeed, to show $\beta$ is one-to-one, note that if $\rmT^x\varphi=\beta(x)=\beta(x')=\rmT^{x'}\varphi$, then $\rmT^{x-x'}\varphi=\varphi$, implying $x-x'\in\calW_\varphi$; since $x$ and $x'$ are both coset representatives of $\calX/\calW_\varphi$, this is a contradiction unless $x=x'$.  Meanwhile, to show $\beta$ is onto, take any $\varphi'\sim\varphi$, and consider a corresponding $x'$ such that $\varphi'=\rmT^{x'}\varphi$.  Taking the unique $x\in\calX/\calW_\varphi$ and $w\in\calW_\varphi$ such that $x'=x+w$, we have: $\varphi'=\rmT^{x'}\varphi=\rmT^{x+w}\varphi=\rmT^x(\rmT^{w}\varphi)=\rmT^x\varphi=\beta(x)$.

Invoking this claim, along with the assumed translation-invariance of $\Phi$, yields:
\begin{equation}
\label{equation.proof of translation-invariant implies flat j2}
\sum_{\varphi'\sim\varphi}\rmP_\Phi(\varphi')1_{\varphi'}
=\sum_{x\in\calX/\calW_\varphi}\rmP_\Phi(\beta(x))1_{\beta(x)}
=\sum_{x\in\calX/\calW_{\varphi}}\rmP_{\Phi}(\rmT^x\varphi)1_{\rmT^x\varphi}
=\sum_{x\in\calX/\calW_{\varphi}}\rmP_{\Phi}(\varphi)1_{\rmT^x\varphi}
=\rmP_{\Phi}(\varphi)\sum_{x\in\calX/\calW_{\varphi}}1_{\rmT^x\varphi}.
\end{equation}
Again, writing any $x'\in\calX$ as $x'=x+w$, where $x\in\calX/\calW_\varphi$ and $w\in\calW_\varphi$, gives:
\begin{equation}
\label{equation.proof of translation-invariant implies flat j3}
\sum_{x'\in\calX}1_{\rmT^{x'}\varphi}
=\sum_{x\in\calX/\calW_{\varphi}}\sum_{w\in\calW_{\varphi}}1_{\rmT^{x+w}\varphi}
=\Biggparen{\sum_{w\in\calW_{\varphi}}1}\sum_{x\in\calX/\calW_{\varphi}}1_{\rmT^x\varphi}
=\abs{\calW_{\varphi}}\sum_{x\in\calX/\calW_{\varphi}}1_{\rmT^x\varphi}.
\end{equation}
Substituting \eqref{equation.proof of translation-invariant implies flat j3} into \eqref{equation.proof of translation-invariant implies flat j2} gives:
\begin{equation}
\label{equation.proof of translation-invariant implies flat j4}
\sum_{\varphi'\sim\varphi}\rmP_\Phi(\varphi')1_{\varphi'}
=\rmP_{\Phi}(\varphi)\sum_{x\in\calX/\calW_{\varphi}}1_{\rmT^x\varphi}
=\frac{\rmP_{\Phi}(\varphi)}{\abs{\calW_\varphi}}\sum_{x'\in\calX}1_{\rmT^{x'}\varphi}.
\end{equation}
Since
$\displaystyle\sum_{x'\in\calX}1_{\rmT^{x'}\varphi}(x,n)
=\sum_{x'\in\calX}\left\{\begin{array}{ll}1,&\varphi(x-x')=n\\0,&\varphi(x-x')\neq n\end{array}\right\}
=\bigabs{\{x'\in\calX:\varphi(x')=n\}}
=\abs{\varphi^{-1}\{n\}},$
substituting \eqref{equation.proof of translation-invariant implies flat j4} into \eqref{equation.proof of translation-invariant implies flat j1} gives:
\begin{equation*}
\overline1_\Phi(x,n)
=\sum_{\varphi\in\calR}\sum_{\varphi'\sim\varphi}\rmP_\Phi(\varphi')1_{\varphi'}(x,n)
=\sum_{\varphi\in\calR}\frac{\rmP_{\Phi}(\varphi)}{\abs{\calW_\varphi}}\sum_{x'\in\calX}1_{\rmT^{x'}\varphi}(x,n)
=\sum_{\varphi\in\calR}\frac{\rmP_{\Phi}(\varphi)}{\abs{\calW_\varphi}}\abs{\varphi^{-1}\{n\}},
\end{equation*}
implying $\Phi$ is flat, since the value of $\overline1_\Phi(x,n)$ depends only on $n$ and is independent of $x$.
\end{proof}
Theorem~\ref{theorem.PJ} indicates that flatness is not too strong of an assumption.  Indeed, one method for producing a flat model $\Phi$ is to generalize the coin-flipping example given in the introduction: given any random method for picking a number from $\bbZ_N$---a probability spinner---produce $\varphi$ by conducting $\abs{\calX}$ independent spins.  The resulting model $\Phi$ is translation-invariant, and therefore flat, since $\rmP_{\Phi}(\varphi)$ is solely determined by the number of times that $\varphi$ achieves each given value $n$.  Other translation-invariant examples abound.  For instance, for any fixed $\varphi_0$, we can assign equal probability \smash{$\frac1{\abs{\calX}}$} to $\varphi_0$ and each of its translates, and assign probability $0$ to all others;  if the source images $\set{f_n}_{n=0}^{N-1}$ are constant, the composite images~\eqref{equation.definition of occlusion} produced by such a model are all translates of a single image.  More generally, we can always partition the $N^{\abs{\calX}}$ elements of $\ell(\calX,\bbZ_N)$ into translation-invariant equivalence classes and assign any fixed probability to the members of each class, provided we ensure that in the end they all sum to one.  For example, for the case $N=2$ and $\calX=\bbZ_2\times\bbZ_2$ depicted in Figure~\ref{figure.ELH}, we may partition the $16$ possible $\varphi$'s into $7$ such classes, and pick any probabilities $\set{\rmp_k}_{k=0}^{15}$ such that $\rmp_1=\rmp_2=\rmp_3=\rmp_4$, $\rmp_5=\rmp_6$, $\rmp_7=\rmp_8$, $\rmp_9=\rmp_{10}$, $\rmp_{11}=\rmp_{12}=\rmp_{13}=\rmp_{14}$.  Armed with one method---translation-invariance---for producing flat models $\Phi$, we now turn to ways of combining known models to produce more complicated and realistic ones.

%%%%%%%%%%%%%%%%%%%%%%%%%%%%%%%%%%%%%%%%%%%%%%%%%%%%%%%%%%%%%%%%
%%%%%%%%%%%%%%%%%%%%%%%%%%%%%%%%%%%%%%%%%%%%%%%%%%%%%%%%%%%%%%%%
% Subsection 4.2: Expansion
%%%%%%%%%%%%%%%%%%%%%%%%%%%%%%%%%%%%%%%%%%%%%%%%%%%%%%%%%%%%%%%%
%%%%%%%%%%%%%%%%%%%%%%%%%%%%%%%%%%%%%%%%%%%%%%%%%%%%%%%%%%%%%%%%

\subsection{Expansion}
Digital microscope images of histological tissues often contain randomly distributed blobs.  These blobs correspond to biological structures: cells, nuclei, etc.  The nature of these processes guarantees that the distribution of such structures is roughly uniform, both spatially and in terms of color: two cells cannot occupy the same space; cells will usually grow and reproduce so as to occupy any empty space; cells in a given tissue all have approximately the same size and color patterns.  We want to construct flat occlusion models that emulate such textures, since in light of Theorem~\ref{theorem.flatness}, doing so would formally justify the demixing of local histograms as part of a segmentation-and-classification algorithm.  Note that there is a natural method for randomly generating a set of roughly uniformly-distributed points: flip a coin at each point $x$.  Here, we explore the idea of \textit{expanding} each of these randomly generated points into a given blob.

To be precise, let $\varphi\in\ell(\calX,\bbZ_2)$ indicate a set of randomly generated points.  For each of the points $x\in\calX$ for which $\varphi(x)=1$, we will replace it with a blob whose shape is indicated by some $\psi_x\in\ell(\calX,\bbZ_2)$.  The new texture will be the union of all such blobs.  Formally, given any $\varphi\in\ell(\calX,\bbZ_2)$ and $\{\psi_x\}_{x\in\calX}\in[\ell(\calX,\bbZ_2)]^\calX$, we define the \textit{expansion} of $\varphi$ by $\{\psi_x\}_{x\in\calX}$ to be $\sconv{\varphi}{\{\psi_x\}_{x\in\calX}}\in\ell(\calX,\bbZ_2)$,
\begin{equation}
\label{equation.StarProduct}
(\sconv{\varphi}{\{\psi_{x'}\}_{x'\in\calX}})(x)
:=\left\{\begin{array}{ll}1,&x=x'+x'',\varphi(x')=1,\psi_{x'}(x'')=1,\\0,&\text{else}.\end{array}\right.
\end{equation}
Two examples of this expansion operation are given in Figure~\ref{figure.Expansion}.
\begin{figure}[t]
\begin{center}
\subfigure[$\varphi$]
{\includegraphics[width=.32\textwidth]{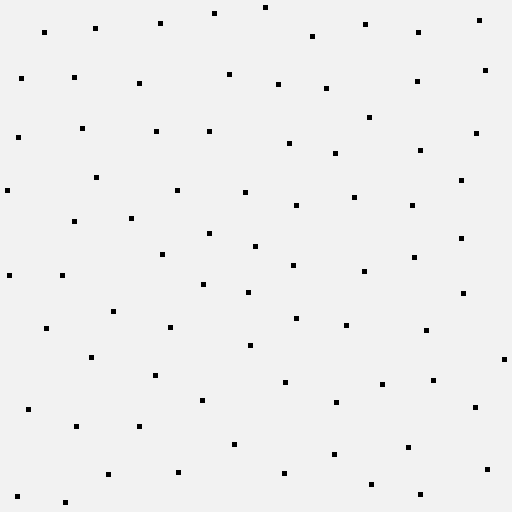}}
\hfill
\subfigure[Some examples of $\psi_x$.]
{\includegraphics[width=.32\textwidth]{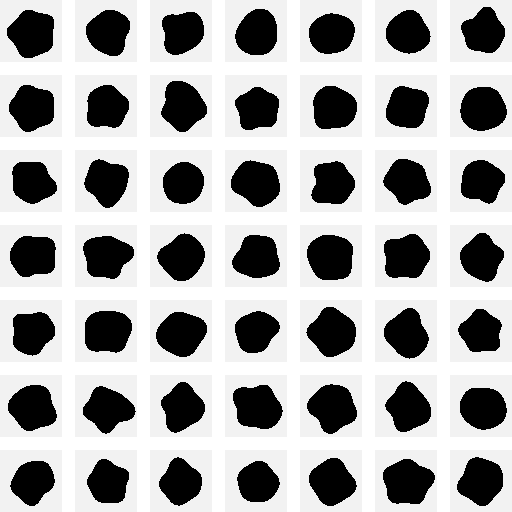}}
\hfill
\subfigure[$\sconv{\varphi}{\{\psi_x\}_{x\in\calX}}$]
{\includegraphics[width=.32\textwidth]{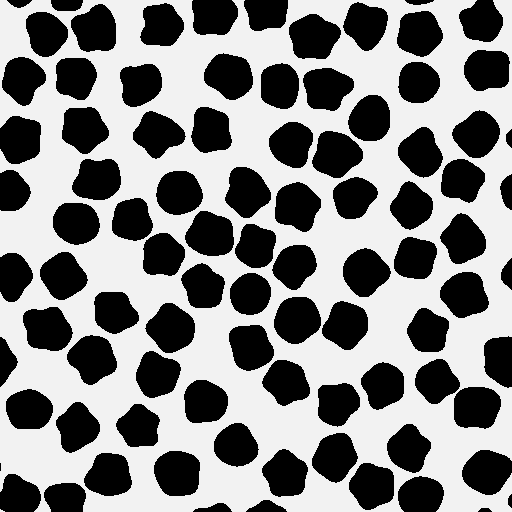}}
\\
\subfigure[$\varphi'$]
{\includegraphics[width=.32\textwidth]{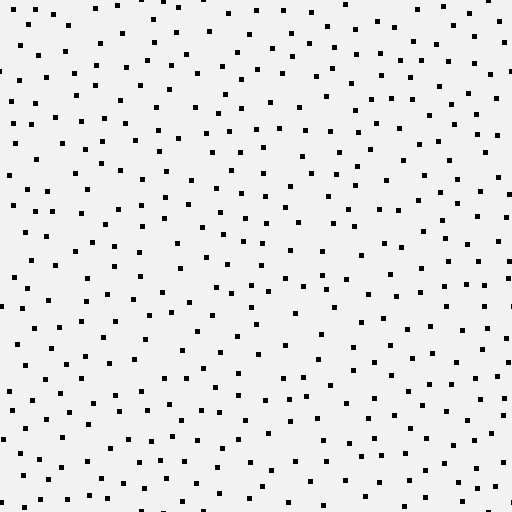}}
\hfill
\subfigure[Some examples of $\psi'_x$.]
{\includegraphics[width=.32\textwidth]{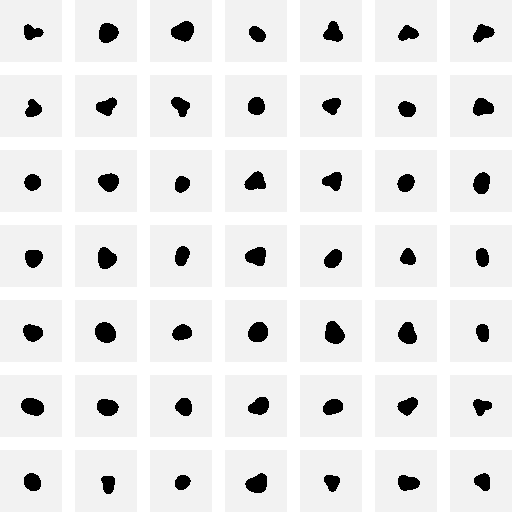}}
\hfill
\subfigure[$\sconv{\varphi'}{\{\psi'_x\}_{x\in\calX}}$]
{\includegraphics[width=.32\textwidth]{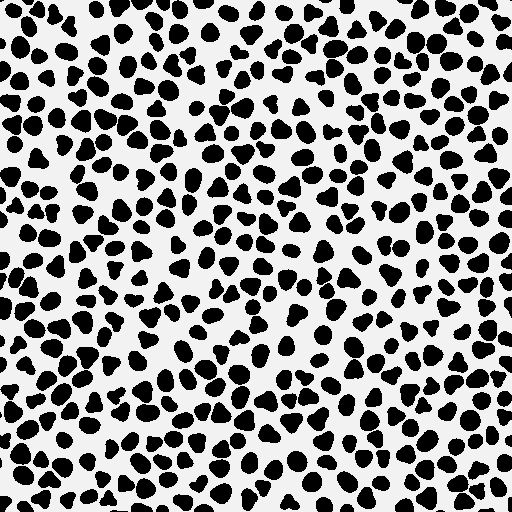}}
\end{center}
\caption{\label{figure.Expansion}Examples of the expansion operation~\eqref{equation.StarProduct}, where black denotes the value of 1, and the lighter shade denotes the value of 0.  A function $\varphi:\calX\rightarrow\set{0,1}$ is given in (a), and can be chosen, for example, via a sequence of $\abs{\calX}$ independent coin flips.  Meanwhile, for each $x\in\calX$, we pick a corresponding function $\psi_x:\calX\rightarrow\set{0,1}$.  Cropped versions of a few examples of such $\psi_x$'s are given in (b).The expansion $\sconv{\varphi}{\{\psi_x\}_{x\in\calX}}$ of $\varphi$ by $\{\psi_x\}_{x\in\calX}$ is given in (c).  Essentially, each point $x$ for which $\varphi(x)=1$ is replaced with the corresponding blob $\psi_x$, with the origin of the $\psi_x$ coordinates being translated to $x$.  In the second row, (f) shows the expansion of a second set of points $\varphi'$ by a second set of blobs $\{\psi'_x\}_{x\in\calX}$.  These examples notwithstanding, note that~\eqref{equation.StarProduct} does not require these blobs to be disjoint.  We could have, for instance, produced a texture by expanding the points in (d) by the blobs in (b).  Nevertheless, stronger conclusions can be made if such disjointness is enforced; see Theorem~\ref{theorem.star product}.}
\end{figure}
Note that expansion itself~\eqref{equation.StarProduct} is not an occlusion model.  Indeed, \eqref{equation.StarProduct} is but a way of combining functions in $\ell(\calX,\bbZ_2)$ to produce other ones, whereas an occlusion model is a random variable $\Phi$ defined by a probability density function $\rmP_\Phi$ over $\ell(\calX,\bbZ_2)$.  This fact notwithstanding, the expansion operation~\eqref{equation.StarProduct} on label functions $\varphi$ and $\{\psi_x\}_{x\in\calX}$ does in fact induce a parallel operation on their random variable cousins $\Phi$ and $\Psi$.  To be precise, given two occlusion models $\Phi$ and $\Psi$ from $\calX$ into $\bbZ_2$, we define the expansion of $\Phi$ by $\Psi$ to be the occlusion model $\sconv{\Phi}{\Psi}$ whose probability density function is $\rmP_{\sconv{\Phi}{\Psi}}:\ell(\calX,\bbZ_2)\rightarrow[0,1]$,
\begin{equation}
\label{equation.StarProductProbability}
\rmP_{\sconv{\Phi}{\Psi}}(\sigma)
:=\sum_{\substack{\varphi\in\ell(\calX,\bbZ_2)\\\{\psi_x\}_{x\in\calX}\in[\ell(\calX,\bbZ_2)]^\calX\\\sconv{\varphi}{\{\psi_x\}_{x\in\calX}}=\sigma}}\rmP_\Phi(\varphi)\prod_{x\in\calX}\rmP_\Psi(\psi_x).
\end{equation}
Note that the probability that $\sconv{\Phi}{\Psi}$ will produce a given label function $\sigma$ depends on the ways in which $\sigma$ can be written as $\sconv{\varphi}{\{\psi_x\}_{x\in\calX}}$ and, moreover, the probability that $\Phi$ and $\Psi$ will produce those particular $\varphi$'s and $\psi_x$'s, respectively.  In the next result, we verify that~\eqref{equation.StarProductProbability} indeed defines a probability density function on $\ell(\calX,\bbZ_2)$.  We further show that if $\Phi$ is translation-invariant~\eqref{equation.definition of translation invariance}, then $\sconv{\Phi}{\Psi}$ is translation-invariant which implies that $\sconv{\Phi}{\Psi}$ is flat by Theorem~\ref{theorem.PJ}.  In particular, image models which produce collections of blobs similar to those found in biological tissues will indeed be flat provided the distribution that produces the ``centers" of these blobs is translation-invariant.  Moreover, if the flatness of $\sconv{\Phi}{\Psi}$ is all that is desired, we can weaken the requirement that $\Phi$ be translation-invariant so as to only require that $\Phi$ is itself flat, provided $\Phi$ and $\Psi$ are \textit{effectively disjoint}:
\begin{equation}
\label{equation.definition of effective disjointness}
\text{If $\rmP_\Phi(\varphi)>0$ and $\rmP_\Psi(\psi_x)>0$ for all $x\in\calX$, then }\sconv{\varphi}{\{\psi_x\}_{x\in\calX}}=\sum_{\substack{x\in\calX\\\varphi(x)=1}}\rmT^{x}\psi_{x}.
\end{equation}
Put another way,~\eqref{equation.definition of effective disjointness} means that there is only at most one way, with nontrivial probability, in which the $x$ in~\eqref{equation.StarProduct} can be written as $x=x'+x''$ where both $\varphi(x')=1$ and $\psi_{x'}(x'')=1$.
\begin{thm}
\label{theorem.star product}
If $\Phi$ and $\Psi$ are occlusion models from $\calX$ into $\bbZ_2$, then their expansion $\sconv{\Phi}{\Psi}$, with probability density function~\eqref{equation.StarProductProbability}, is as well.  Moreover, if $\Phi$ is translation-invariant~\eqref{equation.definition of translation invariance}, then $\sconv{\Phi}{\Psi}$ is translation-invariant.  Furthermore, if $\Phi$ and $\Psi$ are effectively disjoint~\eqref{equation.definition of effective disjointness} and either $\Phi$ or $\Psi$ is flat~\eqref{equation.definition of flatness}, then $\sconv{\Phi}{\Psi}$ is flat.
\end{thm}

\begin{proof}
We first show that~\eqref{equation.StarProductProbability} defines a probability density function, namely that values of $\rmP_{\sconv{\Phi}{\Psi}}(\sigma)$ over all $\sigma$ in $\ell(\calX,\bbZ_2)$ sum to $1$.  Since $\rmP_\Phi$ is a probability density function by assumption, we have:
\begin{equation}
\label{equation.proof of star product 1}
1=\sum_{\varphi\in\ell(\calX,\bbZ_2)}\rmP_\Phi(\varphi).
\end{equation}
Similarly, for any fixed $x\in\calX$, we have:
\begin{equation}
\label{equation.proof of star product 2}
1=\sum_{\psi_x\in\ell(\calX,\bbZ_2)}\rmP_\Psi(\psi_x),
\end{equation}
where the subscript ``$x$" on $\psi$ indicates that this particular $\psi$ is intended to expand $\varphi$ at the particular point $x$ as opposed to at some other point.  Taking the product of~\eqref{equation.proof of star product 1} with the product of~\eqref{equation.proof of star product 2} over all $x$ yields:
\begin{equation}
\label{equation.proof of star product 3}
1
=1(1)^{\abs{\calX}}
=\sum_{\varphi\in\ell(\calX,\bbZ_2)}\rmP_\Phi(\varphi)\prod_{x\in\calX}\,\,\sum_{\psi_x\in\ell(\calX,\bbZ_2)}\rmP_\Psi(\psi_x)
=\sum_{\substack{\varphi\in\ell(\calX,\bbZ_2)\\\{\psi_x\}_{x\in\calX}\in[\ell(\calX,\bbZ_2)]^\calX}}\rmP_\Phi(\varphi)\prod_{x\in\calX}\rmP_\Psi(\psi_x),
\end{equation}
where the final quantity in~\eqref{equation.proof of star product 3} contains all of the cross terms resulting from distributing the product over all sums of the form~\eqref{equation.proof of star product 2}.  Now, since for each choice of $\varphi$ and $\{\psi_x\}_{x\in\calX}$ there is exactly one resulting $\sigma=\sconv{\varphi}{\{\psi_x\}_{x\in\calX}}$, we can rewrite~\eqref{equation.proof of star product 3} in terms of the definition~\eqref{equation.StarProductProbability} of $\rmP_{\sconv{\Phi}{\Psi}}$, obtaining our claim:
\begin{equation*}
1
=\sum_{\sigma\in\ell(\calX,\bbZ_2)}\sum_{\substack{\varphi\in\ell(\calX,\bbZ_2)\\\{\psi_x\}_{x\in\calX}\in[\ell(\calX,\bbZ_2)]^\calX\\\sconv{\varphi}{\{\psi_x\}_{x\in\calX}}=\sigma}}\rmP_\Phi(\varphi)\prod_{x\in\calX}\rmP_\Psi(\psi_x)
=\sum_{\sigma\in\ell(\calX,\bbZ_2)}\rmP_{\sconv{\Phi}{\Psi}}(\sigma).
\end{equation*}
Thus, \eqref{equation.StarProductProbability} indeed defines a probability density function, as claimed.

We next show that the occlusion model $\sconv{\Phi}{\Psi}$ is translation-invariant, if $\Phi$ is translation-invariant.  To do this, we claim that if $\rmT^{\tilde{x}}\sigma=\sconv{\varphi}{\{\psi_x\}_{x\in\calX}}$ then $\sigma=\sconv{(\rmT^{-\tilde{x}}\varphi)}{\{\psi_{x+\tilde{x}}\}_{x\in\calX}}$.  To see this claim, note that 
\begin{equation*}
\sigma(x-\tilde{x})=(\rmT^{\tilde{x}}\sigma)(x)=(\sconv{\varphi}{\{\psi_{x'}\}_{x'\in\calX}})(x)=1
\end{equation*}
if and only if there exists some $x'$, $x''$ in $\calX$ such that $x=x'+x''$, $\varphi(x')=1$, and $\psi_{x'}(x'')=1$.  Letting $\hat{x}=x-\tilde{x}$, we thus have that $\sigma(\hat{x})=1$ if and only if $\hat{x}=(x'-\tilde{x})+x''$, where $(\rmT^{-\tilde{x}}\varphi)(x'-\tilde{x})=\varphi(x'-\tilde{x}+\tilde{x})=\varphi(x')=1$ and $\psi_{(x'-\tilde{x})+\tilde{x}}(x'')=\psi_{x'}(x'')=1$, implying $\sigma=\sconv{(\rmT^{-\tilde{x}}\varphi)}{\{\psi_{x+\tilde{x}}\}_{x\in\calX}}$, as claimed.  Having the claim,~\eqref{equation.StarProductProbability} implies:
\begin{equation*}
\rmP_{\sconv{\Phi}{\Psi}}(\rmT^{\tilde{x}}\sigma)
=\sum_{\substack{\varphi\in\ell(\calX,\bbZ_2)\\\{\psi_x\}_{x\in\calX}\in[\ell(\calX,\bbZ_2)]^\calX\\\sconv{\varphi}{\{\psi_x\}_{x\in\calX}}=\rmT^{\tilde{x}}\sigma}}\rmP_\Phi(\varphi)\prod_{x\in\calX}\rmP_\Psi(\psi_x)
=\sum_{\substack{\varphi\in\ell(\calX,\bbZ_2)\\\{\psi_x\}_{x\in\calX}\in[\ell(\calX,\bbZ_2)]^\calX\\\sconv{(\rmT^{-\tilde{x}}\varphi)}{\{\psi_{x+\tilde{x}}\}_{x\in\calX}}=\sigma}}\rmP_\Phi(\varphi)\prod_{x\in\calX}\rmP_\Psi(\psi_x).
\end{equation*}
To continue, we make the change of variables $\varphi':=\rmT^{-\tilde{x}}\varphi$ and $\psi'_x:=\psi_{x+\tilde{x}}$:
\begin{equation*}
\rmP_{\sconv{\Phi}{\Psi}}(\rmT^{\tilde{x}}\sigma)
=\sum_{\substack{\varphi'\in\ell(\calX,\bbZ_2)\\\{\psi'_x\}_{x\in\calX}\in[\ell(\calX,\bbZ_2)]^\calX\\\sconv{(\varphi')}{\{\psi'_x\}_{x\in\calX}}=\sigma}}\rmP_\Phi(\rmT^{\tilde{x}}\varphi')\prod_{x\in\calX}\rmP_\Psi(\psi'_{x-\tilde{x}}).
\end{equation*}
Since $\Phi$ is translation-invariant and $\displaystyle\prod_{x\in\calX}\rmP_\Psi(\psi'_{x-\tilde{x}})=\prod_{x\in\calX}\rmP_\Psi(\psi'_{x})$, we have:
\begin{equation*}
\rmP_{\sconv{\Phi}{\Psi}}(\rmT^{\tilde{x}}\sigma)
=\sum_{\substack{\varphi'\in\ell(\calX,\bbZ_2)\\\{\psi'_x\}_{x\in\calX}\in[\ell(\calX,\bbZ_2)]^\calX\\\sconv{(\varphi')}{\{\psi'_x\}_{x\in\calX}}=\sigma}}\rmP_\Phi(\varphi')\prod_{x\in\calX}\rmP_\Psi(\psi'_{x})
=\rmP_{\sconv{\Phi}{\Psi}}(\sigma),
\end{equation*}
and so $\sconv{\Phi}{\Psi}$ is indeed translation-invariant~\eqref{equation.definition of translation invariance}, as claimed.

For our final claim, we assume that $\Phi$ and $\Psi$ are effectively disjoint~\eqref{equation.definition of effective disjointness} and that either $\Phi$ or $\Psi$ is flat.  To do so, it is helpful to characterize the flatness of an arbitrary occlusion model $\Phi$ from $\calX$ to $\bbZ_2$ in terms of the corresponding function $\overline\Phi:=\sum_{\varphi\in\ell(\calX,\bbZ_2)}\rmP_\Phi(\varphi)\varphi$.  Indeed, for any $\varphi:\calX\rightarrow\bbZ_2$, \eqref{equation.CharFunction} may be rewritten as $1_\varphi(x,1)=\varphi(x)$ and so:
\begin{equation}
\label{equation.proof of star product 4}
\overline1_\Phi(x,1)
=\sum_{\varphi\in\ell(\calX,\bbZ_2)}\rmP_\Phi(\varphi)1_\varphi(x,1)
=\sum_{\varphi\in\ell(\calX,\bbZ_2)}\rmP_\Phi(\varphi)\varphi(x)
=\overline\Phi(x).
\end{equation}

In light of~\eqref{equation.proof of star product 4}, we claim that $\Phi$ is flat if and only if $\overline\Phi$ is constant.  Indeed, if $\Phi$ is flat, then there exists $\lambda_1$ such that $\overline\Phi(x)=\overline1_\Phi(x,1)=\lambda_1$ for all $x\in\calX$.  Conversely, if $\overline\Phi(x)$ is constant, then there exists $\lambda_1$ such that $\overline1_\Phi(x,1)=\overline\Phi(x)=\lambda_1$ for all $x\in\calX$; by \eqref{equation.expected value of local histograms of occluded images}, this further implies that $\overline1_\Phi(x,0)=1-\overline1_\Phi(x,1)=1-\lambda_1$ for all $x\in\calX$ and so $\Phi$ is flat.

Having this claim, we show that $\sconv{\Phi}{\Psi}$ is flat by showing that $\overline{\sconv{\Phi}{\Psi}}$ is constant.  To do this, we show that if $\Phi$ and $\Psi$ are effectively disjoint then $\overline{\sconv{\Phi}{\Psi}}=\conv{\overline\Phi}{\overline\Psi}$ where ``$*$" denotes standard convolution over $\calX$.  According to the definition of $\sconv{\Phi}{\Psi}$~\eqref{equation.StarProductProbability} we have:
\begin{equation}
\label{equation.proof of star product 5}
\overline{\sconv{\Phi}{\Psi}}
=\sum_{\sigma\in\ell(\calX,\bbZ_2)}\rmP_{\sconv{\Phi}{\Psi}}(\sigma)\sigma
=\sum_{\sigma\in\ell(\calX,\bbZ_2)}\sum_{\substack{\varphi\in\ell(\calX,\bbZ_2)\\\{\psi_x\}_{x\in\calX}\in[\ell(\calX,\bbZ_2)]^\calX\\\sconv{\varphi}{\{\psi_x\}_{x\in\calX}}=\sigma}}\rmP_\Phi(\varphi)\Biggparen{\prod_{x\in\calX}\rmP_\Psi(\psi_x)}(\sconv{\varphi}{\{\psi_x\}_{x\in\calX}}).
\end{equation}
Since any particular choice of $\varphi$ and $\set{\psi_x}_{x\in\calX}$ produces a unique $\sigma$ via $\star$ we can simplify~\eqref{equation.proof of star product 5} to
\begin{equation}
\label{equation.proof of star product 6}
\overline{\sconv{\Phi}{\Psi}}
=\sum_{\substack{\varphi\in\ell(\calX,\bbZ_2)\\\{\psi_x\}_{x\in\calX}\in[\ell(\calX,\bbZ_2)]^\calX}}\rmP_\Phi(\varphi)\Biggparen{\prod_{x\in\calX}\rmP_\Psi(\psi_x)}(\sconv{\varphi}{\{\psi_x\}_{x\in\calX}}).
\end{equation}
Moreover, since $\Phi$ and $\Psi$ are effectively disjoint~\eqref{equation.definition of effective disjointness} we have $\displaystyle\sconv{\varphi}{\{\psi_x\}_{x\in\calX}}=\sum_{\substack{x'\in\calX\\\varphi(x')=1}}\rmT^{x'}\psi_{x'}$ meaning \eqref{equation.proof of star product 6} becomes:
\begin{align}
\nonumber
\overline{\sconv{\Phi}{\Psi}}
&=\sum_{\substack{\varphi\in\ell(\calX,\bbZ_2)\\\{\psi_x\}_{x\in\calX}\in[\ell(\calX,\bbZ_2)]^\calX}}\rmP_\Phi(\varphi)\Biggparen{\prod_{x\in\calX}\rmP_\Psi(\psi_x)}\Biggparen{\sum_{\substack{x'\in\calX\\\varphi(x')=1}}\rmT^{x'}\psi_{x'}}\\
\label{equation.proof of star product 7}
&=\sum_{\varphi\in\ell(\calX,\bbZ_2)}\rmP_\Phi(\varphi)\sum_{\substack{x'\in\calX\\\varphi(x')=1}}\rmT^{x'}\Biggbracket{\sum_{\{\psi_x\}_{x\in\calX}\in[\ell(\calX,\bbZ_2)]^\calX}\Biggparen{\prod_{x\in\calX}\rmP_\Psi(\psi_x)}\psi_{x'}}.
\end{align}
Now, for any fixed $x'\in\calX$ such that $\varphi(x')=1$, we factor the corresponding innermost sum in~\eqref{equation.proof of star product 7} into a product of $\abs{\calX}$ distinct sums---one for each $x\in\calX$---to obtain:
\begin{equation}
\label{equation.proof of star product 8}
\sum_{\{\psi_x\}_{x\in\calX}\in[\ell(\calX,\bbZ_2)]^\calX}\Biggparen{\prod_{x\in\calX}\rmP_\Psi(\psi_x)}\psi_{x'}
=\Biggbracket{\prod_{x\neq x'}\Biggparen{\sum_{\psi_x\in\ell(\calX,\bbZ_2)}\rmP_\Psi(\psi_x)}}\sum_{\psi_{x'}\in\ell(\calX,\bbZ_2)}\rmP_\Psi(\psi_{x'})\psi_{x'}
=\Biggparen{\prod_{x\neq x'}1}\overline\Psi
=\overline\Psi.
\end{equation}
Substituting \eqref{equation.proof of star product 8} into \eqref{equation.proof of star product 7} then gives:
\begin{equation*}
\overline{\sconv{\Phi}{\Psi}}
=\sum_{\varphi\in\ell(\calX,\bbZ_2)}\rmP_\Phi(\varphi)\sum_{\substack{x'\in\calX\\\varphi(x')=1}}\rmT^{x'}\overline\Psi
=\sum_{\varphi\in\ell(\calX,\bbZ_2)}\rmP_\Phi(\varphi)\conv{\Biggparen{\sum_{\substack{x'\in\calX\\\varphi(x')=1}}\delta_{x'}}}{\overline\Psi}
=\conv{\Biggparen{\sum_{\varphi\in\ell(\calX,\bbZ_2)}\rmP_\Phi(\varphi)\varphi}}{\overline\Psi}
=\conv{\overline\Phi}{\overline\Psi}.
\end{equation*}
Thus, the effective disjointness of $\Phi$ and $\Psi$ indeed implies $\overline{\sconv{\Phi}{\Psi}}=\conv{\overline\Phi}{\overline\Psi}$.  As such, if we further assume that either $\Phi$ or $\Psi$ is flat, then either $\overline\Phi$ or $\overline\Psi$ is constant, implying in either case that $\overline{\sconv{\Phi}{\Psi}}$ is constant and so $\sconv{\Phi}{\Psi}$ is flat.
\end{proof}

%%%%%%%%%%%%%%%%%%%%%%%%%%%%%%%%%%%%%%%%%%%%%%%%%%%%%%%%%%%%%%%%%%%%%%%%%%%%%%%%%%%%%%%%%%%%%%%%%%%%%
%%%%%%%%%%%%%%%%%%%%%%%%%%%%%%%%%%%%%%%%%%%%%%%%%%%%%%%%%%%%%%%%%%%%%%%%%%%%%%%%%%%%%%%%%%%%%%%%%%%%%
% Subsection 4.2: Overlay
%%%%%%%%%%%%%%%%%%%%%%%%%%%%%%%%%%%%%%%%%%%%%%%%%%%%%%%%%%%%%%%%%%%%%%%%%%%%%%%%%%%%%%%%%%%%%%%%%%%%%
%%%%%%%%%%%%%%%%%%%%%%%%%%%%%%%%%%%%%%%%%%%%%%%%%%%%%%%%%%%%%%%%%%%%%%%%%%%%%%%%%%%%%%%%%%%%%%%%%%%%%

\subsection{Overlay}

Above, we discussed how the expansion~\eqref{equation.StarProductProbability} of a binary-valued occlusion model $\Phi$ with another such model $\Psi$ is a new model $\sconv{\Phi}{\Psi}$ that randomly generates label functions of the form $\sigma=\sconv{\varphi}{\{\psi_x\}_{x\in\calX}}$ as defined in~\eqref{equation.StarProduct}.  Under certain hypotheses, Theorem~\ref{theorem.star product} gives that such models $\sconv{\Phi}{\Psi}$ are flat, meaning their local histograms can be understood in terms of Theorem~\ref{theorem.flatness}.  Moreover, some examples of these models produce textures that resemble those encountered in histological tissues: if $ f_0$ and $ f_1$ are roughly constant light purple and dark purple fields, respectively, then the composite image $\occ{\varphi}\set{ f_0, f_1}$ obtained by picking $\varphi$ as in Figure~\ref{figure.Expansion}(f) bears some similarity to an actual image of cartilage, such as the one given in Figure~\ref{figure.histograms}(b).  Taken together, these facts provide some theoretical justification for the use of local histograms for the analysis of such tissues.

There is however a deficit with this theory: due to the nature of the construction~\eqref{equation.StarProduct}, models produced by expansion~\eqref{equation.StarProductProbability} can only be binary-valued, and as such are insufficient to emulate textures that exhibit three or more distinct color modes, such as the pseudovascular tissue depicted in Figure~\ref{figure.histograms}(d).  In this subsection, we discuss a method for \textit{laying} one occlusion model over another which, amongst other things, permits us to build multivalued models out of binary-valued ones.  To be precise, for any $\varphi\in\ell(\calX,\bbZ_{N_{\varphi}})$, $\psi\in\ell(\calX,\bbZ_{N_{\psi}})$ and $\sigma\in\ell(\calX,\bbZ_2)$, we define the \textit{overlay} of $\varphi$ over $\psi$ with respect to $\sigma$ to be $\varphi\#_{\sigma}\psi\in\ell(\calX,\bbZ_{N_{\varphi}+N_{\psi}})$,
\begin{equation}
\label{equation.definition of overlay}
(\varphi\#_{\sigma}\psi)(x)
:=\left\{\begin{array}{ll}\varphi(x),&\sigma(x)=0,\\\psi(x)+N_{\varphi},&\sigma(x)=1.\end{array}\right.
\end{equation}
Essentially, an overlay~\eqref{equation.definition of overlay} is the result of cutting holes out of an image of $\varphi$ and laying it on top of an image of $\psi$; the location of these holes is indicated by $\sigma$ and the values of $\psi$ are increased by a factor of $N_\varphi$ so that they cannot be confused with those of $\varphi$.  Examples of this overlay operation are given in Figure~\ref{figure.Overlay}.

\begin{figure}[t]
\begin{center}
\subfigure[$\sigma$]
{\includegraphics[width=.33\textwidth]{6c.png}}
\hspace{.045\textwidth}
\subfigure[$0\#_{\sigma}\sigma'$]
{\includegraphics[width=.33\textwidth]{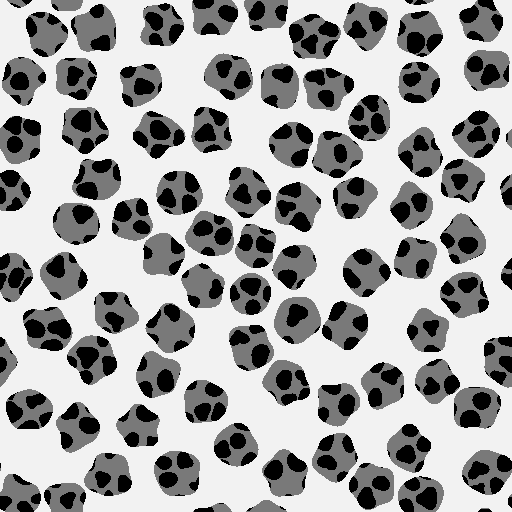}}
\\
\subfigure[$\sigma'$]
{\includegraphics[width=.33\textwidth]{6f.png}}
\hspace{.045\textwidth}
\subfigure[$\sigma\#_{\sigma'}0$]
{\includegraphics[width=.33\textwidth]{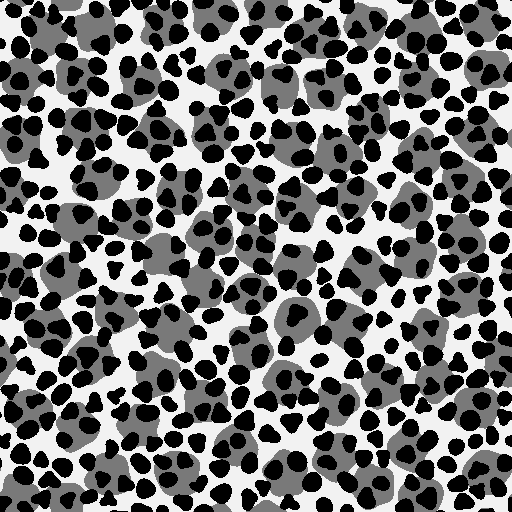}}
\end{center}
\caption{\label{figure.Overlay}Two examples of the overlay operation~\eqref{equation.definition of overlay}.  Recall the two $\set{0,1}$-valued label functions $\sigma=\sconv{\varphi}{\{\psi_x\}_{x\in\calX}}$ and $\sigma'=\sconv{\varphi'}{\{\psi'_x\}_{x\in\calX}}$ of Figure~\ref{figure.Expansion}(c) and (f) reshown here in (a) and (c), respectively.  Further consider a constant function $0:\calX\rightarrow\bbZ_1$ that assigns label $0$ to every point in $x$.  The overlay~\eqref{equation.definition of overlay} of $0$ over $\sigma'$ is given in (b); essentially, $\sigma$-shaped holes are cut from $0$ and the result is laid over $\sigma'$, resulting in a new texture.  A distinct texture can be produced by cutting $\sigma'$-shaped holes out from $\sigma$ and laying the result over the constant function $0$ (d).  Overlaying the resulting textures with each other can produce even more complicated textures.}
\end{figure}
In a manner similar to the relationship between~\eqref{equation.StarProduct} and~\eqref{equation.StarProductProbability}, we have that~\eqref{equation.definition of overlay} naturally induces a parallel operation on occlusion models: given probability density functions $\rmP_\Phi$, $\rmP_\Psi$ and $\rmP_\Sigma$ on $\ell(X,\bbZ_{N_\Phi})$, $\ell(X,\bbZ_{N_\Psi})$ and $\ell(X,\bbZ_2)$, respectively, we define the overlay of the occlusion model $\Phi$ over $\Psi$ with respect to $\Sigma$ to be the new occlusion model $\Phi\#_{\Sigma}\Psi$ whose probability density function is $\rmP_{\Phi\#_\Sigma\Psi}:\ell(\calX,\bbZ_{N_{\varphi}+N_{\psi}})\rightarrow[0,1]$,
\begin{equation}
\label{equation.ProbOfPoundProduct}
\rmP_{\Phi\#_\Sigma\Psi}(\upsilon)
:=\sum_{\substack{\varphi\in\ell(\calX,\bbZ_{N_{\varphi}})\\\psi\in\ell(\calX,\bbZ_{N_{\psi}})\\\sigma\in\ell(\calX,\bbZ_2)\\\varphi\#_{\sigma}\psi=\upsilon}}\rmP_\Phi(\varphi)\rmP_\Psi(\psi)\rmP_\Sigma(\sigma).
\end{equation}
In the next result, we verify that~\eqref{equation.ProbOfPoundProduct} indeed defines a probability density function, and moreover that the corresponding model $\Phi\#_{\Sigma}\Psi$ is flat provided $\Phi$, $\Psi$ and $\Sigma$ are flat, meaning that the local histograms~\eqref{equation.definition of local histogram} of composite images~\eqref{equation.definition of occlusion} produced by such a model will behave according to Theorem~\ref{theorem.flatness}.

\begin{thm}
\label{theorem.overlay}
If $\Phi$, $\Psi$ and $\Sigma$ are occlusion models on $\ell(X,\bbZ_{N_\Phi})$, $\ell(X,\bbZ_{N_\Psi})$ and $\ell(X,\bbZ_2)$, respectively, then~\eqref{equation.ProbOfPoundProduct} defines a probability density function on $\ell(\calX,\bbZ_{N_\varphi+N_\psi})$.  Moreover,  if $\Phi$, $\Psi$, and $\Sigma$ are flat, then $\Phi\#_\Sigma\Psi$ is flat.
\end{thm}
\begin{proof}
To show that \eqref{equation.ProbOfPoundProduct} defines a probability density function on $\ell(\calX,\bbZ_{N_\varphi+N_\psi})$, note that:
\begin{equation}
\label{equation.proof of overlay 1}
1
=(1)(1)(1)
=\Biggparen{\,\sum_{\varphi\in\ell(\calX,\bbZ_{N_\varphi})}\rmP_\Phi(\varphi)}\Biggparen{\,\sum_{\psi\in\ell(\calX,\bbZ_{N_\psi})}\rmP_\Psi(\psi)}\Biggparen{\,\sum_{\sigma\in\ell(\calX,\bbZ_2)}\rmP_\Sigma(\sigma)}
=\sum_{\substack{\varphi\in\ell(\calX,\bbZ_{N_{\varphi}})\\\psi\in\ell(\calX,\bbZ_{N_{\psi}})\\\sigma\in\ell(\calX,\bbZ_2)}}\rmP_\Phi(\varphi)\rmP_\Psi(\psi)\rmP_\Sigma(\sigma).
\end{equation}
Noting that for each fixed $\varphi$, $\psi$, and $\sigma$, there exists exactly one $\upsilon\in\ell(\calX,\bbZ_{N_\varphi+N_\psi})$ such that $\varphi\#_\sigma\psi=\upsilon$, \eqref{equation.proof of overlay 1} becomes:
\begin{equation*}
1
=\sum_{\upsilon\in\ell(\calX,\bbZ_{N_\varphi+N_\psi})}\sum_{\substack{\varphi\in\ell(\calX,\bbZ_{N_{\varphi}})\\\psi\in\ell(\calX,\bbZ_{N_{\psi}})\\\sigma\in\ell(\calX,\bbZ_2)\\\varphi\#_\sigma\psi=\upsilon}}\rmP_\Phi(\varphi)\rmP_\Psi(\psi)\rmP_\Sigma(\sigma)
=\sum_{\upsilon\in\ell(\calX,\bbZ_{N_\varphi+N_\psi})}\rmP_{\Phi\#_\Sigma\Psi}(\upsilon),
\end{equation*}
as claimed.  For the second conclusion, assume that $\Phi$, $\Psi$, and $\Sigma$ are flat.  Our goal is to show that $\Phi\#_\Sigma\Psi$ is flat~\eqref{equation.definition of flatness}, meaning that for any $n\in\bbZ_{N_\varphi+N_\psi}$, we want to show that there exists a scalar $\lambda_n$ such that:
\begin{equation}
\label{equation.proof of overlay 2}
\sum_{\substack{\upsilon\in\ell(\calX,\bbZ_{N_\varphi+N_\psi})\\\upsilon(x)=n}}\rmP_{\Phi\#_\Sigma\Psi}(\upsilon)
=\lambda_n
\end{equation}
for all $x\in\calX$.   To see this, note that for any such $x$ and $n$, we have:
\begin{equation}
\label{equation.proof of overlay 3}
\sum_{\substack{\upsilon\in\ell(\calX,\bbZ_{N_\varphi+N_\psi})\\\upsilon(x)=n}}\rmP_{\Phi\#_\Sigma\Psi}(\upsilon)
=\sum_{\substack{\upsilon\in\ell(\calX,\bbZ_{N_\varphi+N_\psi})\\\upsilon(x)=n}}\sum_{\substack{\varphi\in\ell(\calX,\bbZ_{N_\varphi})\\\psi\in\ell(\calX,\bbZ_{N_\psi})\\\sigma\in\ell(\calX,\bbZ_2)\\\varphi\#_\sigma\psi=\upsilon}}\rmP_\Phi(\varphi)\rmP_\Psi(\psi)\rmP_\Sigma(\sigma)
=\sum_{\substack{\varphi\in\ell(\calX,\bbZ_{N_\varphi})\\\psi\in\ell(\calX,\bbZ_{N_\psi})\\\sigma\in\ell(\calX,\bbZ_2)\\(\varphi\#_\sigma\psi)(x)=n}}\rmP_\Phi(\varphi)\rmP_\Psi(\psi)\rmP_\Sigma(\sigma).
\end{equation}
Now, in the special case where $n=0,\ldots,N_\varphi-1$, \eqref{equation.definition of overlay} gives that $(\varphi\#_\sigma\psi)(x)=n$ if and only if $\varphi(x)=n$ and $\sigma(x)=0$.  As such, in this case \eqref{equation.proof of overlay 3} becomes:
\begin{align}
\nonumber
\sum_{\substack{\upsilon\in\ell(\calX,\bbZ_{N_\varphi+N_\psi})\\\upsilon(x)=n}}\rmP_{\Phi\#_\Sigma\Psi}(\upsilon)
&=\sum_{\substack{\varphi\in\ell(\calX,\bbZ_{N_\varphi}),\varphi(x)=n\\\psi\in\ell(\calX,\bbZ_{N_\psi})\\\sigma\in\ell(\calX,\bbZ_2),\sigma(x)=0}}\rmP_\Phi(\varphi)\rmP_\Psi(\psi)\rmP_\Sigma(\sigma)\\
\nonumber
&=\Biggparen{\,\sum_{\substack{\varphi\in\ell(\calX,\bbZ_{N_\varphi})\\\varphi(x)=n}}\rmP_\Phi(\varphi)}\Biggparen{\,\sum_{\psi\in\ell(\calX,\bbZ_{N_\psi})}\rmP_\Psi(\psi)}\Biggparen{\,\sum_{\substack{\sigma\in\ell(\calX,\bbZ_2)\\\sigma(x)=0}}\rmP_\Sigma(\sigma)}\\
\label{equation.proof of overlay 4}
&=\lambda_{\Phi,n}\lambda_{\Sigma,0}.
\end{align}
If, on the other hand $n=N_\varphi,\ldots,N_\varphi+N_\psi-1$ then \eqref{equation.definition of overlay} gives that $(\varphi\#_\sigma\psi)(x)=n$ if and only if $\psi(x)=n-N_\varphi$ and $\sigma(x)=1$.  In this case, \eqref{equation.proof of overlay 3} becomes:
\begin{align}
\nonumber
\sum_{\substack{\upsilon\in\ell(\calX,\bbZ_{N_\varphi+N_\psi})\\\upsilon(x)=n}}\rmP_{\Phi\#_\Sigma\Psi}(\upsilon)
&=\sum_{\substack{\varphi\in\ell(\calX,\bbZ_{N_\varphi})\\\psi\in\ell(\calX,\bbZ_{N_\psi}),\psi(x)=n-N_\varphi\\\sigma\in\ell(\calX,\bbZ_2),\sigma(x)=1}}\rmP_\Phi(\varphi)\rmP_\Psi(\psi)\rmP_\Sigma(\sigma)\\
\nonumber
&=\Biggparen{\,\sum_{\varphi\in\ell(\calX,\bbZ_{N_\varphi})}\rmP_\Phi(\varphi)}\Biggparen{\,\sum_{\substack{\psi\in\ell(\calX,\bbZ_{N_\psi})\\\psi(x)=n-N_\varphi}}\rmP_\Psi(\psi)}\Biggparen{\,\sum_{\substack{\sigma\in\ell(\calX,\bbZ_2)\\\sigma(x)=1}}\rmP_\Sigma(\sigma)}\\
\label{equation.proof of overlay 5}
&=\lambda_{\Psi,n-N_\varphi}\lambda_{\Sigma,1}.
\end{align}
Thus, for any $x\in\calX$ we either have~\eqref{equation.proof of overlay 4} or~\eqref{equation.proof of overlay 5} meaning $\Phi\#_\Sigma\Psi$ is flat~\eqref{equation.proof of overlay 2}, as claimed.
\end{proof}

%%%%%%%%%%%%%%%%%%%%%%%%%%%%%%%%%%%%%%%%%%%%%%%%%%%%%%%%%%%%%%%%
%%%%%%%%%%%%%%%%%%%%%%%%%%%%%%%%%%%%%%%%%%%%%%%%%%%%%%%%%%%%%%%%
% Section 5: A local histogram-based segmentation-and-classification algorithm
%%%%%%%%%%%%%%%%%%%%%%%%%%%%%%%%%%%%%%%%%%%%%%%%%%%%%%%%%%%%%%%%
%%%%%%%%%%%%%%%%%%%%%%%%%%%%%%%%%%%%%%%%%%%%%%%%%%%%%%%%%%%%%%%%

\section{A local histogram-based segmentation-and-classification algorithm}
\label{section.MajorVeinOne}

In this section, we present a proof-of-concept segmentation-and-classification scheme that is inspired by Theorem~\ref{theorem.flatness}.  We emphasize that for the algorithm presented here, local histograms are the only image features that are computed.  That is, the decision of which label to assign to a given pixel is based purely on the distribution of color in its surrounding neighborhood.  We do this to demonstrate the validity of the concept embodied by Theorem~\ref{theorem.flatness} as an image processing tool.  For algorithms intended for real-world use, such color information should be combined with morphological data---size, local and global shape, orientation and organization---in order to obtain better classification accuracies.  An example of such an algorithm, accompanied by thorough testing and comparisons against other state-of-the-art methods, is given in the sister article~\cite{Bhagavatula10Nov} to this one; these facts are not reprinted here.

The concept of Theorem~\ref{theorem.flatness} is that the local histograms of certain textures should, on the whole, be able to be decomposed in terms of more basic distributions.  Indeed, it is reasonable to expect a local histogram computed over a region of cartilage (Figure~\ref{figure.histograms}(b)) to be a mixture of $0.8$ of a ``light purple" distribution---a distribution mostly supported in portions of $\calY$ that correspond to light purple---with $0.2$ of a darker reddish-purple one.  Meanwhile, local histograms of other tissues will correspond to distinct mixtures of other distributions.  For example, local histograms computed over a region of pseudovascular tissue (Figure~\ref{figure.histograms}(d)) might be a mixture of $0.5$ of a light pink distribution, with $0.25$ of a dark purple one and $0.25$ of a reddish-pink one.

The algorithm we present here exploits this concept.  The first step is to train our classifier.  To do so, let $K$ be the number of distinct tissue types found in a training image such as Figure~\ref{figure.classification}(a) or (d).  For any tissue type $k=1,\dotsc,K$, we compute local histograms $\set{h_{k;m}}_{m=1}^{M_k}$ about pixel locations $\set{x_{k;m}}_{m=1}^{M_k}$ that have been labeled as being of that type by medical experts.  Each $h_{k;m}$ is a nonnegatively-valued function over $\calY$ that sums to one.  There are several ways to pick the $x_{k;m}$'s.  One approach is to have the expert choose each point individually.  Alternatively, if the expert has manually segmented and labeled the entire image (Figure~\ref{figure.classification}(b)), then the $x_{k;m}$'s can be chosen at random from regions of type $k$.  The number $M_k$ of local histograms that we compute for type $k$ is somewhat arbitrary; we used repeated experimentation to find a sample size large enough to guarantee reliably-decent performance.
\begin{figure}
\begin{center}
\subfigure[]
{\includegraphics[width=.32\textwidth]{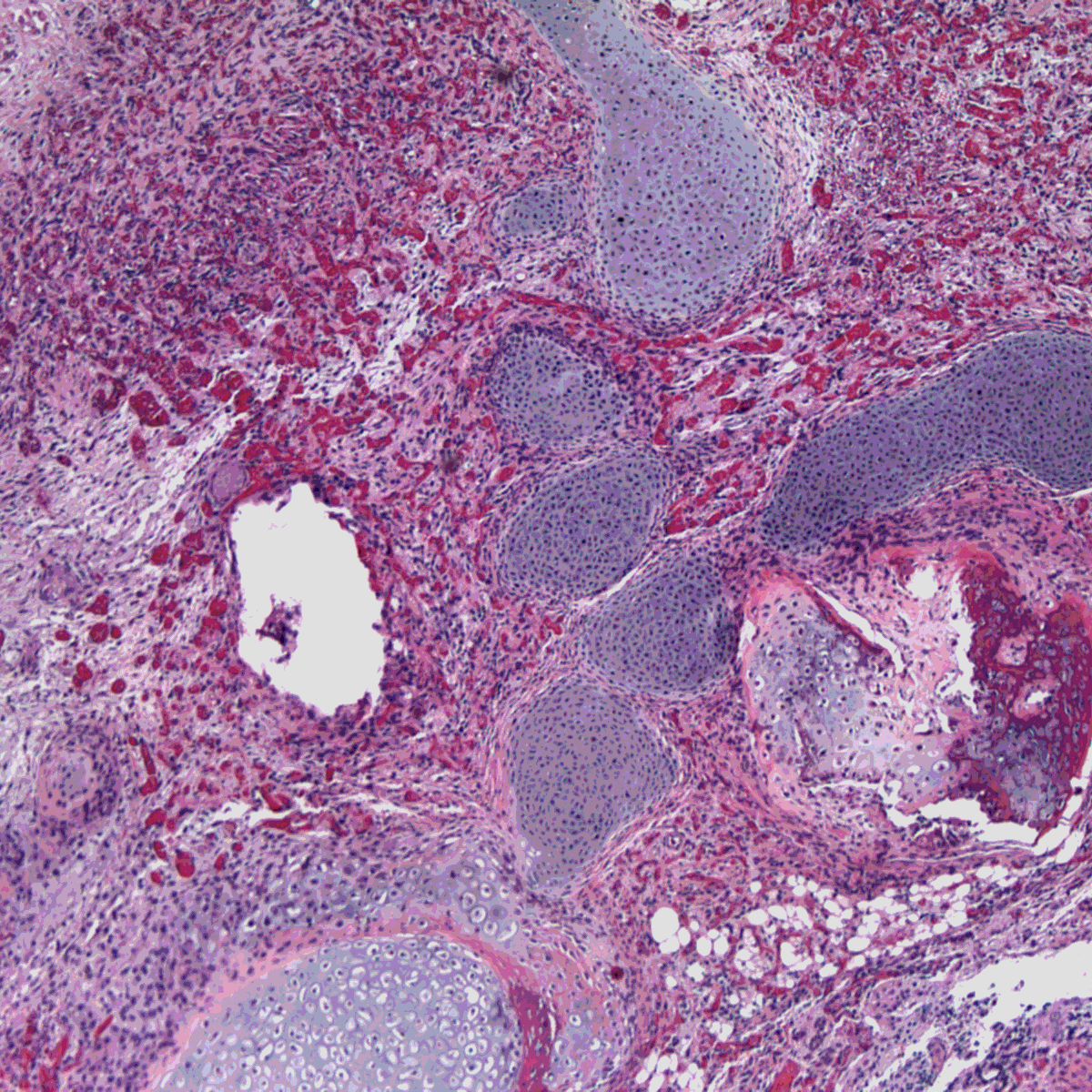}}
\hfill
\subfigure[]
{\includegraphics[width=.32\textwidth]{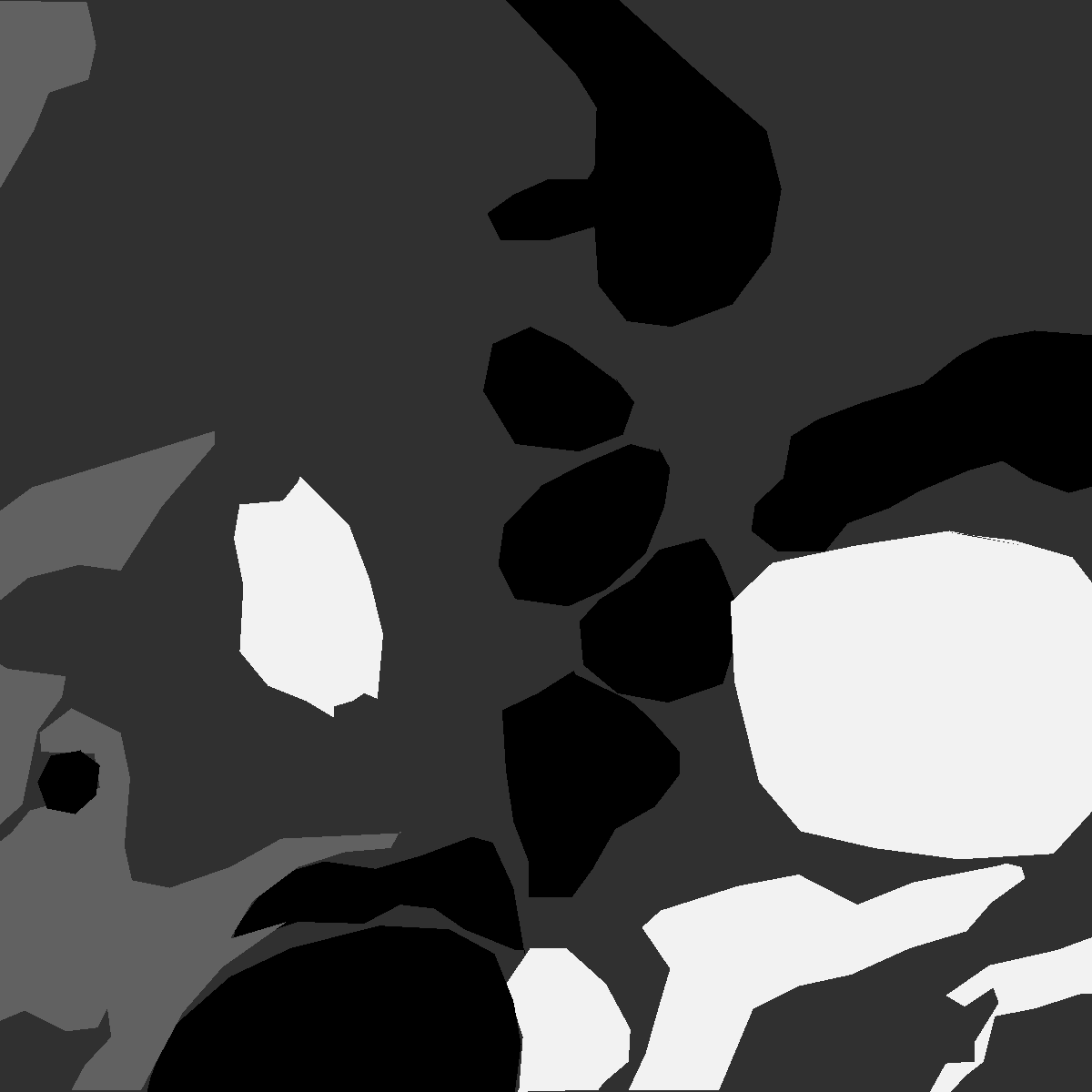}}
\hfill
\subfigure[]
{\includegraphics[width=.32\textwidth]{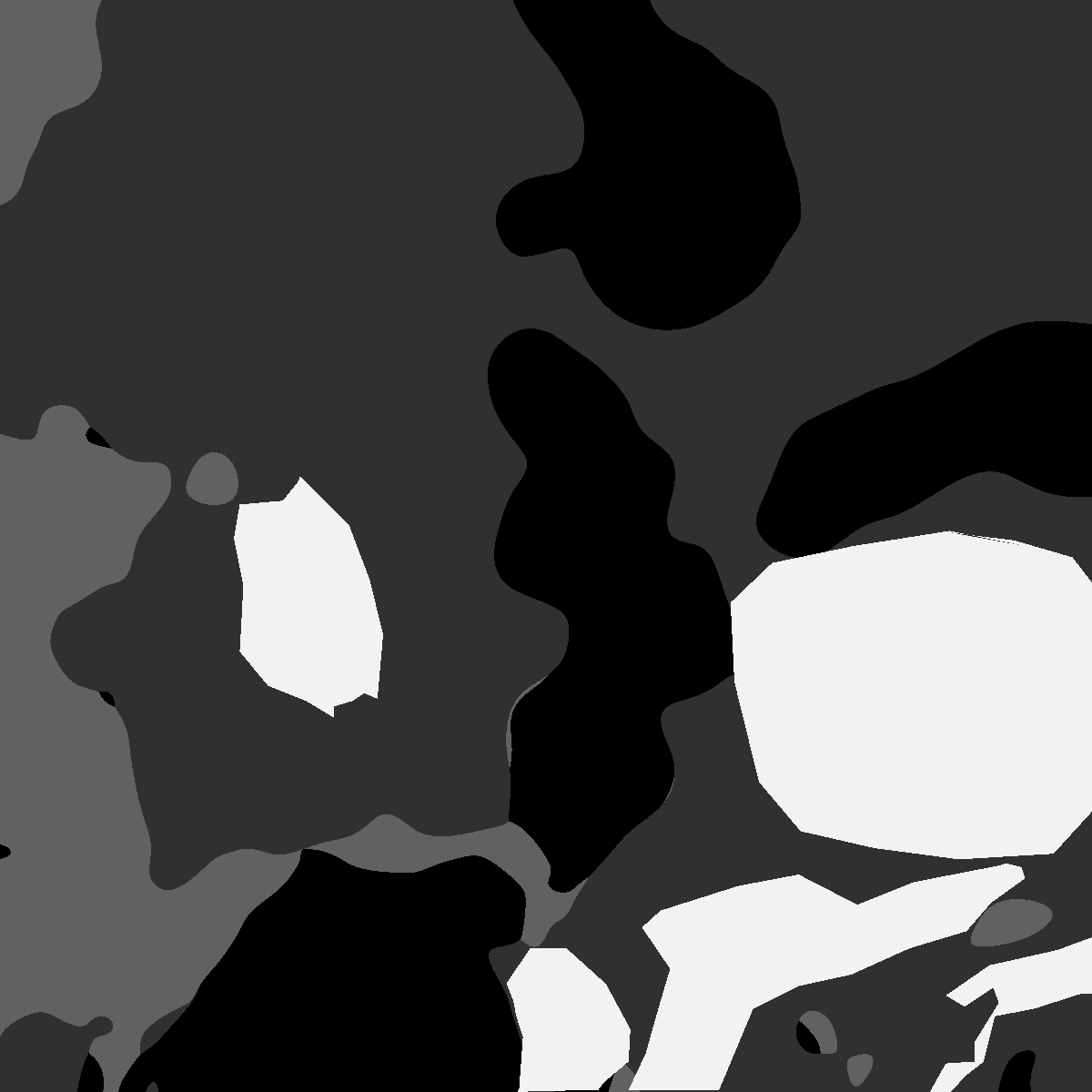}}
\\
\subfigure[]
{\includegraphics[width=.32\textwidth]{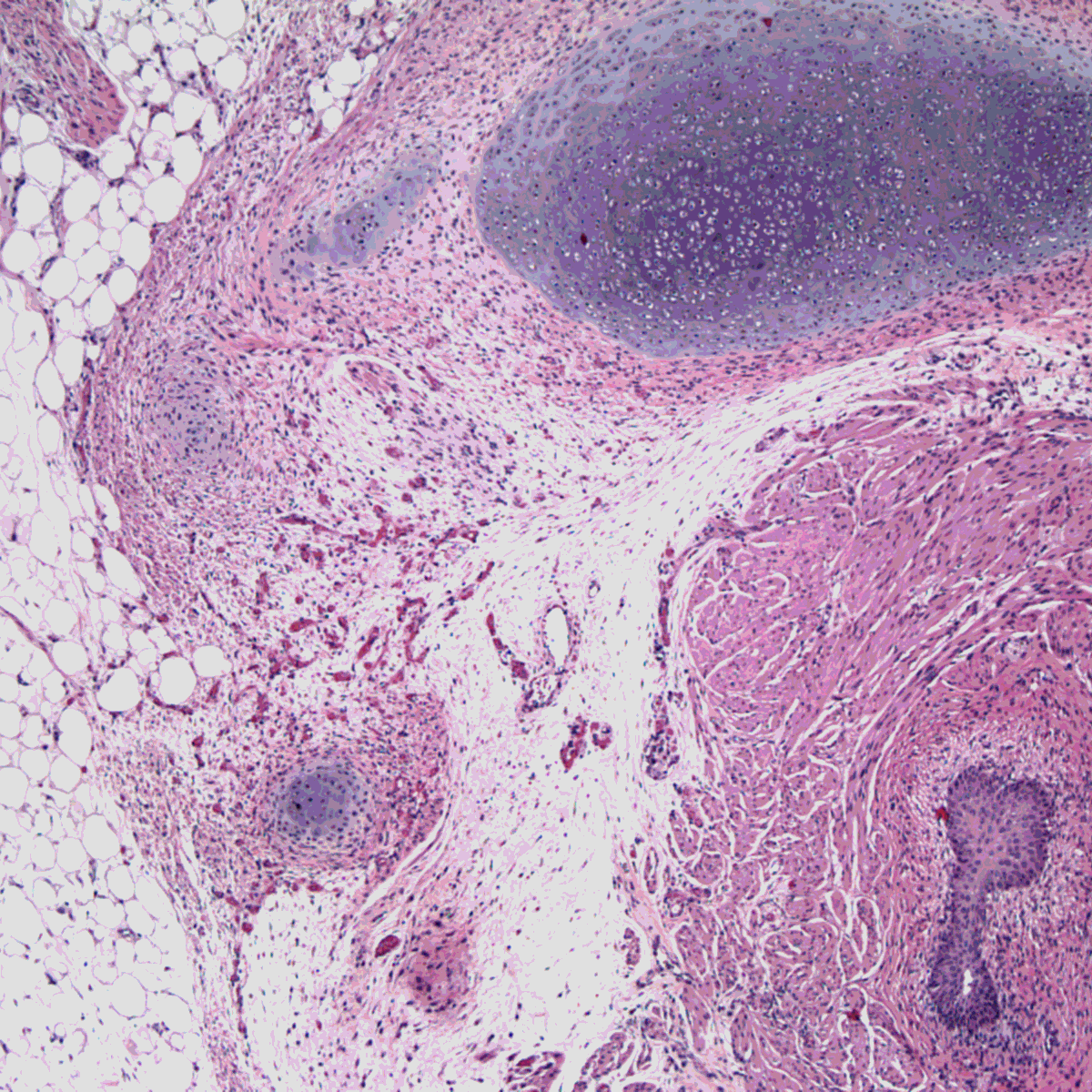}}
\hfill
\subfigure[]
{\includegraphics[width=.32\textwidth]{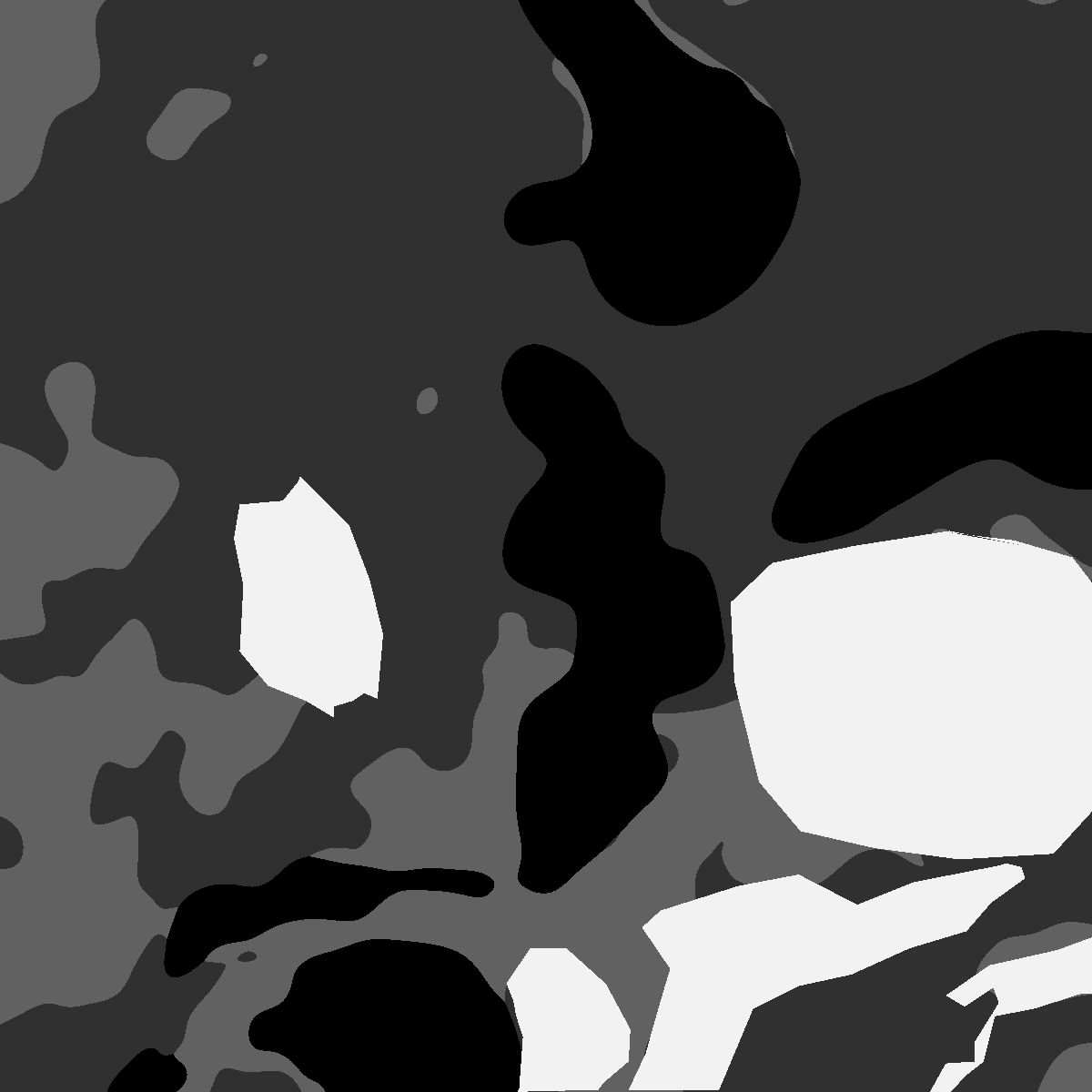}}
\end{center}
\caption{\label{figure.classification} An example of using PCA of local histograms to perform segmentation and classification of the image given in (a), which is a 3-bit quantized version of Figure~\ref{figure.Histology}(a).  A manually segmented and labeled version of (a) is shown in (b) where black represents cartilage, light gray represents connective tissue, dark gray represents pseudovascular tissue, and white represents other tissues that have been ignored in this proof-of-concept experimentation.  Using (a) as both the training and testing image in a PCA-based classification scheme~\eqref{equation.PCA decision rule}, we obtain the labels shown in (c).  A similar, but less-accurate classification of (a) can still be obtained if we instead train on (d), resulting in the labels given in (e).
}
\end{figure}

In light of Theorem~\ref{theorem.flatness}, it would be nice to demix the training local histograms $\set{h_{k;m}}_{m=1}^{M_k}$ in terms of a type-dependent class of more basic distributions \smash{$\set{g_{k;n}}_{n=1}^{N_k}$}.  That is, we would like to find nonnegatively-valued functions $g_{k;n}$ over $\calY$ that sum to one and have the property that for each training local histogram $h_{k;m}$ there exists nonnegative scalars $\set{\lambda_{k;m,n}}_{n=1}^{N_k}$ that themselves sum to one and such that:
\begin{equation}
\label{equation.demixing}
h_{k;m}\approx\sum_{n=1}^{N_k}\lambda_{k;m,n} g_{k;n},\quad\forall m=1,\dotsc,M_k.
\end{equation}
Unfortunately, computing the $g_{k;n}$'s that minimize the approximation error in~\eqref{equation.demixing} is a nontrivial optimization problem.  As such, we leave this approach for future work, and instead consider a mathematically-simpler problem in which the $\lambda_{k;m,n}$'s and $g_{k;n}$'s are permitted to be arbitrary real scalars and vectors, respectively.  That is, we perform PCA for each tissue type $k$.  To be precise, for each type, we form a $\abs{\calY}\times M_k$ matrix $H_k$ whose columns are the (vectorized) local histograms $h_{k;m}$ less their average $\overline{h}_k$: 
\begin{equation}
\label{equation.definition of covariance matrix}
H_k(:,m)=h_{k;m}-\overline{h}_k\qquad\text{where}\qquad\overline{h}_{k}=\frac1{M_k}\sum_{m=1}^{M_k}h_{k;m}.
\end{equation}
We then compute the singular value decompositions $H_k=U_k^{}\Sigma_k^{}V_k^{\mathrm{T}}$ and identify those left-singular vectors \smash{$\set{u_{k;n}}_{n=1}^{N_k}$} that correspond to some experimentally-determined number $N_k$ of dominant singular values $\set{\sigma_{k;n}}_{n=1}^{N_k}$.  In this setting, the approximation~\eqref{equation.demixing} is replaced by:
\begin{equation}
\label{equation.PCA training}
h_{k;m}\approx\overline{h}_k+\sum_{n=1}^{N_k}\ip{h_{k;m}-\overline{h}_k}{u_{k;n}}u_{k;n},\quad\forall m=1,\dotsc,M_k.
\end{equation}
The classical theory of PCA states that the approximation error in~\eqref{equation.PCA training} is optimally small in the sense that these specific $u_{k;n}$'s span the particular $N_k$-dimensional subspace of $\ell(\calY,\bbR)$ whose orthogonal projection operator $P_k$ minimizes the total squared-error $\sum_{m=1}^{M_k}\norm{h_{k;m}-\overline{h}_k-P_k(h_{k;m}-\overline{h}_k)}^2$.  The vectors $\overline{h}_k$ and \smash{$\set{u_{k;n}}_{n=1}^{N_k}$} in hand, we store them in memory, completing the training phase of our classification algorithm.

To segment and label a given image $f$, we compute its local histograms~\eqref{equation.definition of local histogram}, obtaining local distributions of color $h_x:\calY\rightarrow\bbR$, $h_x(y)=(\LH_ w f)(x,y)$ about every pixel location $x\in\calX$.  At any given $x$, we then assign a tissue label $k(x)$ by finding the tissue type $k$ whose shifted subspace $\overline{h}_k+\mathrm{span}\set{u_{k;n}}_{n=1}^{N_k}$ is nearest to $h_x$.  Specifically, we let:
\begin{align}
\nonumber
k(x)
&=\underset{k=1,\dotsc,K}{\mathrm{argmin}}\,\Bignorm{h_x-\overline{h}_k-\sum_{n=1}^{N_k}\ip{h_x-\overline{h}_k}{u_{k;n}}u_{k;n}}^2\\
\nonumber
&=\underset{k=1,\dotsc,K}{\mathrm{argmin}}\,\Bigparen{\norm{h_x-\overline{h}_k}^2-\sum_{n=1}^{N_k}\abs{\ip{h_x-\overline{h}_k}{u_{k;n}}}^2}\\
\label{equation.PCA decision rule}
&=\underset{k=1,\dotsc,K}{\mathrm{argmin}}\,\Biggparen{\,\sum_{y\in\calY}\abs{h_x(y)-\overline{h}_k(y)}^2-\sum_{n=1}^{N_k}\biggabs{\sum_{y\in\calY}\bigparen{h_x(y)-\overline{h}_k(y)}u_{k;n}(y)}^2}.
\end{align}
In implementation, we compute the summations over $\calY$ in~\eqref{equation.PCA decision rule} as running sums, looping over all $y\in\calY$.  This computational trick greatly reduces our memory requirements: at any given time, we only store a single level of $\LH_ w f$.  By Theorem~\ref{theorem.PropertiesOfLHs}, such a level can be obtained by filtering an indicator function; in the following experimental results, we avoided edge artifacts by using a weighted noncyclic method of filtering, namely the $\star$-convolution of~\cite{Srinivasa09}.  Without such a trick, one must store the entire local histogram transform in memory, a daunting task for even modestly-sized images: the full local histogram transform of the $1200\times1200$, 8-bit RGB image given in Figure~\ref{figure.Histology}(a) is a $1200\times1200\times256\times256\times256$ array.

Further computational advantages may be gained by quantizing the image and reducing the dimension of the color space.  For our particular set of histology images, we experimentally found that we could still obtain good accuracies even if we discard the green channel of our purple-pink images, and moreover quantize the $8$-bit red and blue channels down to $3$-bits apiece.  That is, we quantize $\calY$ from $\bbZ_{256}^3$ to $\bbZ_8^2$.  By Proposition~\ref{proposition.LHresults}, this is equivalent to binning the original $1200\times1200\times256\times256\times256$ local histogram array down to a new one of size $1200\times1200\times8\times8$.  The quantized version of Figure~\ref{figure.Histology}(a) is given in Figure~\ref{figure.classification}(a); for the sake of readability, a $3$-bit quanitized version of the unused green channel was included in this rendering.  As a result of this quantization, it only takes a few seconds to assign per-pixel labels to a $1200\times1200$ histology image using a MATLAB-based implementation of~\eqref{equation.PCA decision rule}, running on standard desktop hardware.  For this particular set of images, further color quantization, such as using $2$-bit colors ($\calY=\bbZ_2^2$) or converting the original image to grayscale ($\calY=\bbZ_{256}$), results in an unacceptable loss in classification accuracy, as do attempts at spatial quantization ($\calX=\bbZ_{600}^2$).  

Two runs of this classification algorithm are depicted in Figure~\ref{figure.classification}.  In the first run, we train the classifier on the $3$-bit $1200\times1200$ red-blue image given in~Figure~\ref{figure.classification}(a).  For the sake of simplicity, we restrict ourselves to $K=3$ tissue types: cartilage, connective tissue and pseudovascular tissue; all other tissue types are ignored in the confusion matrices given below.  For each type $k=1,2,3$, we randomly choose $M_k=64$ points of that type, making use of a small number of the $1200^2$ ground truth labels given in~Figure~\ref{figure.classification}(b); edge artifacts are avoided by not picking points near the border.  For each type, we then perform PCA on the $64$ local histograms $h_{k;m}$ of that type, computing an average local histogram $\overline{h}_k$ as well as the dominant left-singular vectors of $H_k$~\eqref{equation.definition of covariance matrix}.  For the sake of simplicity, in a given experiment we will use the same number of principal components for each of the three types, that is, $N_k=N$ for $k=1,2,3$.  At the same time, we experiment with this number itself, letting $N$ be either $1$, $2$, $3$ or $4$.  With the training complete, we then segment and classify Figure~\ref{figure.classification}(a) using the decision rule~\eqref{equation.PCA decision rule}, resulting in per-pixel labels such as the ones given in Figure~\ref{figure.classification}(c) for $N=4$.  Comparing Figure~\ref{figure.classification}(c) and the ground truth of Figure~\ref{figure.classification}(b), we see both the power and limitations of local histograms: color is a big factor in determining tissue type, but by ignoring shape, we suffer from oversmoothing.  The accuracy percentages for various choices of $N$ are given by a confusion matrix:
\begin{equation*}
\begin{tabular}{lrrrcrrrcrrrcrrr }
&\multicolumn{3}{c}{$N=1$}&&\multicolumn{3}{c}{$N=2$}&&\multicolumn{3}{c}{$N=3$}&&\multicolumn{3}{c}{$N=4$}\\
\hline
	&Ca &Co	&Ps	&	&Ca &Co	&Ps	&	&Ca &Co	&Ps	&	&Ca &Co	&Ps\\
\hline								
Ca  &77 &22	&1  &	&87 &11	&2	&	&96 &3	&1	&	&96 &3	&1 \\
Co  &0  &95	&5  &	&0  &91	&9	&	&3  &94	&3	&	&3  &92	&5\\
Ps  &2  &11	&87 & 	&2  &8	&90	&	&6  &7	&87	&	&5  &5	&90\\
\end{tabular}
\end{equation*}
Here each row of the matrix tells us the percentage a certain tissue was labeled as cartilage (Ca), connective tissue (Co), and pseudovascular tissue (Ps).  In particular, the first three entries of the first row of this table tell us that when using a single principal component, those points labeled as cartilage by a medical expert in Figure~\ref{figure.classification}(b) are correctly labeled as such by our algorithm $77\%$ of the time, while $22\%$ of it is mislabeled as connective tissue and $1\%$ of it is mislabeled as pseudovascular tissue.  Note here that we have trained and tested on the same image; such experiments indicate the feasibility of our approach in a semi-automated classification scheme in which a medical expert handpicks $64$ points of each given type and lets the algorithm automatically assign labels to the rest.

The second run of this algorithm is almost identical to the first, with the exception that we use a distinct image in the training phase.  To be precise, for each of the three tissue types, we perform PCA on the local histograms of $64$ randomly-chosen points of that type in Figure~\ref{figure.classification}(d), making use of its ground truth labels (not pictured).  We then apply the principal components obtained from Figure~\ref{figure.classification}(d) to generate labels (Figure~\ref{figure.classification}(e)) for Figure~\ref{figure.classification}(a) using the decision rule~\eqref{equation.PCA decision rule}.  Compared to the first run, the algorithm's performance here is a better indication of its feasibility as a fully automated classification scheme, and is summarized by the following confusion matrix:
\begin{equation*}
\begin{tabular}{lrrrcrrrcrrrcrrr }
&\multicolumn{3}{c}{$N=1$}&&\multicolumn{3}{c}{$N=2$}&&\multicolumn{3}{c}{$N=3$}&&\multicolumn{3}{c}{$N=4$}\\
\hline
	&Ca &Co	&Ps	&	&Ca &Co	&Ps	&	&Ca &Co	&Ps	&	&Ca &Co	&Ps\\
\hline								
Ca  &90 &9	&1  &	&91 &4	&5	&	&90 &5	&5	&	&83 &11	&6\\
Co  &25 &61	&14 &	&10 &62	&28	&	&7  &79	&14	&	&8  &70	&22\\
Ps  &30 &12	&58 & 	&4  &10	&86	&	&4  &50	&46	&	&2  &17	&81\\
\end{tabular}
\end{equation*}
Though the performance in the second run is understandably worse than that of the first, it nevertheless demonstrates the real-world potential of the idea exemplified by Theorem~\ref{theorem.flatness}: the local histograms of certain types of textures can be decomposed into more basic distributions, and this decomposition can serve as an image processing tool.

\section*{Acknowledgments}
The authors are extremely grateful to Dr.~Carlos Castro and Dr.~John A.~Ozolek for introducing us to the motivating application and for providing us with raw image data and manually segmented ground truth labels used throughout this article.  This work is supported by NSF DMS 1042701 and CCF 1017278, AFOSR F1ATA01103J001 and F1ATA00183G003NIH, NIH-R03-EB009875 and 5P01HD047675-02, the PA State Tobacco Settlement and the Kamlet-Smith Bioinformatics Grant.  Parts of the work were presented at ISBI 2010~\cite{Bhagavatula10} and SBEC 2010~\cite{Massar10}.  The views expressed in this article are those of the authors and do not reflect the official policy or position of the United States Air Force, Department of Defense, or the U.S.~Government.

\end{document}